\documentclass[reqno]{amsart}  

\theoremstyle{plain}
\newtheorem{theorem}{Theorem}[section]

\newtheorem{remark}[theorem]{Remark}
\newtheorem{proposition}[theorem]{Proposition}

\newtheorem{corollary}[theorem]{Corollary}

\numberwithin{equation}{section}
\allowdisplaybreaks[1]

\theoremstyle{definition}

\theoremstyle{remark}

\newcommand{\bU}{{\mathbf U}}

\newcommand{\cC}{{\mathcal C}}

\newcommand{\cE}{{\mathcal E}}

\newcommand{\cG}{{\mathcal G}}

\newcommand{\cI}{{\mathcal I}}

\newcommand{\cM}{{\mathcal M}}
\newcommand{\cN}{{\mathcal N}}
\newcommand{\cR}{{\mathcal R}}
\newcommand{\cS}{{\mathcal S}}

\newcommand{\cU}{{\mathcal U}}
\newcommand{\bu}{{\mathbf u}}

\newcommand{\bbD}{{\mathbb D}}

\newcommand{\bbZ}{{\mathbb Z}}
\newcommand{\C}{{\mathbb C}}

\newcommand{\sbm}[1]{\left[\begin{smallmatrix} #1
          \end{smallmatrix}\right]}

\newcommand{\mat}[2]{\ensuremath{\left[\begin{array}{#1}
#2
\end{array} \right]}}
\newcommand{\Pick}{{\mathbb P}}
\newcommand{\Q}{{\mathbb Q}}
\newcommand{\tilE}{{\widetilde E}}
\newcommand{\tilX}{{\widetilde X}}
\newcommand{\tilW}{{\widetilde W}}
\newcommand{\tilK}{{\widetilde K}}
\newcommand{\half}{{\frac{1}{2}}}

\begin{document}

\title[Constrained Nevanlinna-Pick interpolation]{A constrained
Nevanlinna-Pick interpolation problem for matrix-valued functions }
\author[J. A. Ball]{Joseph A. Ball}
\address{Department of Mathematics,
Virginia Tech,
Blacksburg, VA 24061-0123, USA}
\email{ball@math.vt.edu}
\author[V. Bolotnikov]{Vladimir Bolotnikov}
\address{Department of Mathematics,
The College of William and Mary,
Williamsburg VA 23187-8795, USA}
\email{vladi@math.wm.edu}
\author[S. ter Horst]{Sanne ter Horst}
\address{Department of Mathematics,
Virginia Tech,
Blacksburg, VA 24061-0123, USA}
\email{terhorst@math.vt.edu}

\begin{abstract}
    Recent results of Davidson-Paulsen-Raghupathi-Singh
    give necessary and sufficient conditions for the existence of a
    solution to the Nevanlinna-Pick interpolation problem on the unit
    disk with the additional restriction that the interpolant should have
    the value of its derivative at the origin equal to zero.  This
    concrete mild generalization of the classical problem is prototypical
    of a number of other generalized Nevanlinna-Pick interpolation
    problems which have appeared in the literature (for example, on a
    finitely-connected planar domain or on the polydisk).  We extend the
    results of Davidson-Paulsen-Raghupathi-Singh to the setting where
    the interpolant is allowed to be matrix-valued and elaborate further
    on the analogy with the theory of Nevanlinna-Pick interpolation on a
    finitely-connected planar domain.

\end{abstract}


\maketitle

\section{Introduction}

The classical Nevanlinna-Pick interpolation problem \cite{P16,N19}
has a data set
\begin{equation} \label{clas-data}
{\mathfrak D}: (z_{1}, w_{1}), \dots, (z_{n}, w_{n})
\end{equation}
where $(z_{1}, \dots, z_{n})$ is an $n$-tuple of distinct points in
the unit disk ${\mathbb D} = \{ \lambda \in {\mathbb C} \colon$
$|\lambda| < 1\}$ and $(w_{1}, \dots, w_{n})$ is a collection of
complex numbers in ${\mathbb C}$, and asks for the existence of a
{\em Schur-class function} $s$, i.e., a holomorphic function $s$
mapping the open unit disk ${\mathbb D}$ into the closed unit disk
$\overline{\mathbb D}$, such that $s$ satisfies the set of
interpolation conditions associated with the data set ${\mathfrak
D}$:
\begin{equation}  \label{clas-int}
       s(z_{j}) = w_{j}\quad \text{for}\quad j = 1, \dots, n.
\end{equation}
The well-known existence criterion \cite{P16} is the following: {\em
there exists a Schur-class function $s$ satisfying the interpolation
conditions \eqref{clas-int} if and only if the associated Pick
matrix given by
\begin{equation}\label{pick-clas}
      {\mathbb P}  = \left[ \frac{ 1 - w_{i} \overline{w}_{j}}{1 - z_{i}
     \overline{z}_{j}} \right]_{i,j = 1, \dots, n}
\end{equation}
is positive-semidefinite.}  In the recent paper \cite{DPRS},
Davidson-Paulsen-Raghupathi-Singh considered a variant of the
classical problem, where the Schur class ${\mathcal S}$ is replaced
by the constrained Schur-class ${\mathcal S}_{1}$ consisting of
Schur-class functions $s$ satisfying the additional constraint $s'(0) =
0$. It is readily checked that ${\mathcal
S}_{1}$ is the unit ball of the  Banach algebra $H^{\infty}_{1}=
\{f\in H^\infty({\mathbb D}): \; f'(0) =0\}$. The constrained
Nevanlinna-Pick
interpolation problem considered in \cite{DPRS} then is:

\medskip

{\bf CNP:} {\em Find a function $s\in{\mathcal S}_{1}$
satisfying interpolation conditions \eqref{clas-int}.}

\medskip

In case $z_{j} = 0$ for some $j$, one can treat the problem as a
standard Carath\'eodory-Fej\'er interpolation problem with
interpolation condition on the derivative at $0$:
\begin{equation}  \label{clas-int'}
s(z_{j}) = w_{j}\quad\text{for}\quad j = 1, \dots, n, \quad s'(0) = 0
\end{equation}
and the existence criterion has the same form: solutions exist if
and only if a slightly modified Pick matrix ${\mathbb P}$ is
positive-semidefinite.  If no $z_{j}$ is equal to $0$, the
solvability criterion is given in \cite{DPRS} as follows. With a
pair $(\alpha, \beta)$ of complex numbers we associate the positive
kernel on ${\mathbb D}\times {\mathbb D}$
\begin{equation}  \label{scalar-ker}
         K^{\alpha, \beta}(z,w) = (\alpha + \beta z) \overline{(\alpha +
         \beta w)} + \frac{ z^{2} \overline{w^{2}}}{1 - z \overline{w}}.
\end{equation}

\begin{theorem}
\label{T:1DPRS}
The problem {\bf CNP} has a solution if and only if the matrix
     \begin{equation}  \label{clas-crit1}
     \left[ (1 - w_{i} \overline{w_{j}}) K^{\alpha, \beta}(z_i,z_j)
     \right]_{i,j = 1, \dots, n}
     \end{equation}
is positive-semidefinite for each choice of $(\alpha, \beta)\in\C^2$
satisfying $|\alpha|^{2} + |\beta|^{2} = 1$.
\end{theorem}
In practice in Theorem \ref{T:1DPRS} it suffices to work with pairs
for which $\alpha \ne 0$. The paper \cite{DPRS} gives a second solution
criterion for the problem {\bf CNP}.
\begin{theorem}  \label{T:2DPRS} The problem {\bf CNP} has a solution if
and only if there exists $\lambda \in {\mathbb D}$ such that the
matrix
\begin{equation}  \label{clas-crit2}
    \left[ \frac{ z_{i}^{2} \overline{z}_{j}^{2} -
    \varphi_{\lambda}(w_{i}) \overline{ \varphi_{\lambda}(w_{j})}}{1 -
    z_{i} \overline{z_{j}}} \right]_{i,j = 1, \dots, n}
\end{equation}
is positive-semidefinite, where $\varphi_{\lambda}(z) =
{\displaystyle\frac{z - \lambda}{1 - \overline{\lambda} z}}$.
\end{theorem}

As pointed out in \cite{DPRS}, the criteria \eqref{clas-crit1} and
\eqref{clas-crit2} are complementary in the following sense. Using
\eqref{clas-crit1} to check that a solution exists appears not to be
practical since one must check the positive-semidefiniteness of
infinitely many matrices; on the other hand, to check that a solution
does not exists is a finite test (if one is lucky): exhibit one
admissible parameter $(\alpha, \beta)$ such that the associated Pick
matrix $ (1 - w_{i} \overline{w_{j}})K^{\alpha, \beta}(z_i, z_{j})$
is not positive-semidefinite.  The situation with the criterion
\eqref{clas-crit2} is the reverse.  To check that a solution exists
using criterion \eqref{clas-crit2} is a finite test: one need only
be lucky enough to find a single $\lambda$ for which the associated
matrix \eqref{clas-crit2} is positive-semidefinite.  On the other
hand, the check via \eqref{clas-crit2} that a solution does not exist
requires checking the lack of positive-semidefiniteness for an infinite
family of Pick matrices.

The paper \cite{DPRS} also considers the matrix-valued version of
the constrained Nevan\-linna-Pick problem.  Here one specifies a
data set
\begin{equation}  \label{op-data}
    {\mathfrak D}: (z_{1}, W_{1}), \dots, (z_{n}, W_{n})
\end{equation}
where $z_{1}, \dots, z_{n}$ are distinct nonzero points in ${\mathbb D}$ as
before, but where now $W_{1}$, $\dots$, $W_{n}$ are $k \times k$ complex
matrices. We let $(\cS_{1})^{k \times k}$ be the $k \times k$
matrix-valued constrained Schur-class
$$
(\cS_{1})^{k \times k} = \{ S \in (H^{\infty}({\mathbb D}))^{k \times k}
\colon \| S \|_{\infty} \le 1 \text{ and } S^\prime(0) = 0 \}.
$$
Then the $k \times k$ matrix-valued constrained Nevanlinna-Pick
problem is:
\medskip

{\bf MCNP:} {\em Find $S \in (\cS_{1})^{k \times k}$
satisfying interpolation conditions
\begin{equation}  \label{op-int}
     S(z_{j}) = W_{j}\quad\text{for}\quad j = 1, \dots, n.
\end{equation}}
It is not difficult to show that a {\em necessary} condition for the
existence of a solution to the problem ${\bf MCNP}$ is that the matrix
\begin{equation}  \label{nec-pos}
     \left[ (I_{k} - W_{i} W_{j}^{*})  K^{\alpha, \beta}(z_{i},
     z_{j}) \right]_{i,j = 1, \dots, n}
\end{equation}
be positive-semidefinite for all pairs of complex numbers $(\alpha,
\beta)$ with $|\alpha|^{2} + | \beta|^{2} = 1$.  The main result in
\cite{DPRS} on the problem {\bf MCNP} is that this condition in general is
not sufficient:  {\em there exist three distinct nonzero points
$z_1,z_2,z_3$ in ${\mathbb D}$ together with three $k \times k$ matrices
$W_{1}, W_{2},
W_{3}$ for some integer $k > 1$ so that the $3 \times 3$ block matrix $[(I_{k} -
    W_{i} W_{j}^{*}) K^{\alpha, \beta}(z_{i}, z_{j})]$ is positive
    semidefinite for all choices of $(\alpha, \beta)$ with $|\alpha|^{2}
    + | \beta|^{2} = 1$ yet there is no function $S \in (\cS_1)^{k\times
k}$ so that $S(z_{i}) = W_{i}$ for $i = 1, 2, 3.$}
One of the main results of the present paper is that nevertheless Theorem
\ref{T:1DPRS} can be recovered for the matrix-valued case if one
enriches the collection $\left[ (I - W_{i}W_{j}^{*})K^{\alpha,
\beta}(z_{i}, z_{j}) \right]$ of Pick matrices to be checked for
positive-semidefiniteness.

To state our results, we introduce some notation. For $\ell, \ell'$
a pair of integers satisfying $1 \le \ell \le \ell' \le k$,  we let
${\mathbb G}(\ell' \times \ell)$ be the set of pairs $(\alpha,
\beta)$ of $\ell' \times \ell$ matrices subject to the constraints
that $\alpha \alpha^{*} + \beta \beta^{*} = I_{\ell'}$ and that
$\alpha$ be injective. If two elements $(\alpha,\beta)$ and
$(\alpha',\beta')$ of ${\mathbb G}(\ell' \times \ell)$ are related
by $\alpha' = U\alpha$, $\beta' = U\beta$ for a $\ell' \times\ell'$
unitary matrix $U$, then the kernel functions $K^{\alpha, \beta}$
and $K^{\alpha', \beta'}$ (defined as in \eqref{RK} below) are the
same.  Thus what is of interest are the {\em equivalence classes} of
elements of ${\mathbb G}(\ell' \times \ell)$ induced by this
equivalence relation $(\alpha, \beta) \equiv (\alpha', \beta')$.  If
we drop the restriction that $\alpha$ is injective, then such
equivalence classes are in one-to-one correspondence with the {\em
Grassmannian} (as suggested vaguely by our notation ${\mathbb
G}(\ell' \times  \ell)$) of $\ell'$-dimensional subspaces of $ 2
\ell$-dimensional space; with the injectivity restriction on
$\alpha$ in place, ${\mathbb G}(\ell' \times \ell)$ still
corresponds to a dense subset of this Grassmannian.  This is the
canonical generalization of the analysis in \cite{DPRS} where the
special case $\ell' = \ell  = 1$ appears in the same context.

Given any $(\alpha, \beta) \in {\mathbb G}(\ell' \times \ell)$, we
define the $\ell \times \ell$-matrix kernel function $K^{\alpha,
\beta}(z,w)$ on ${\mathbb D} \times {\mathbb D}$ with values in
${\mathbb C}^{\ell \times \ell}$ by
\begin{equation}  \label{RK}
       K^{\alpha, \beta}(z,w) = (\alpha^{*} + \overline{w} \beta^{*})
       (\alpha + z \beta) + \frac{ \overline{w^{2}} z^{2}}{1 -
       \overline{w} z} I_{\ell}.
\end{equation}
    Then we have the following result.

    \begin{theorem}  \label{T:1}
The problem {\bf MCNP} has a solution if and only if for each $(\alpha,
\beta)\in {\mathbb  G}(\ell' \times \ell)$ and for all $n$-tuples $(X_{1},
\dots, X_{n})$
        of  $k \times \ell$ matrices $X_{1}, \dots, X_{n}$ where  $1 \le
        \ell \le \ell' \le k$, it is the case that
        \begin{equation}  \label{op-ker-family}
       \sum_{i,j=1}^{n} \operatorname{trace} \left[
     X_{j} K^{\alpha, \beta}(z_{i}, z_{j}) X_{i}^{*} -
        W_{j}^{*} X_{j} K^{\alpha, \beta}(z_{i}, z_{j}) X_{i}^{*}
        W_{i} \right] \ge 0.
        \end{equation}
      \end{theorem}

      In case $k=1$, we have $\ell' =
\ell =
      1$ and then ${\mathbb G}(1 \times 1)$ consist of pairs of
complex numbers
      $(\alpha, \beta)$ subject to $|\alpha|^{2} + |\beta|^{2} = 1$ and
      the associated kernel $K^{\alpha, \beta}$ is as in
      \eqref{scalar-ker}.  Then the quantities $X_{1}, \dots, X_{n}$ are
      just complex numbers and it is straightforward to see that the
      condition in Theorem \ref{T:1} collapses to the
      positive-semidefiniteness of the single matrix in
      \eqref{clas-crit1} for each $(\alpha, \beta) \in {\mathbb G}(1
      \times 1)$.
      In this way we see that Theorem \ref{T:1} contains Theorem
      \ref{T:1DPRS} as a corollary.  For the general case of Theorem
      \ref{T:1}, we see that restriction of Theorem \ref{T:1} to the
      case where $\alpha,
      \beta$ are scalar $k \times k$ matrices leads to the necessity of
      the positive-semidefiniteness of the matrix  \eqref{nec-pos}.

      We also have an extension of Theorem
      \ref{T:2DPRS} to the matrix-valued setting.  Given a data set
      ${\mathfrak D}$ as in \eqref{op-data} with no $z_{i}$ equal to $0$,
      the Pick matrix associated with the standard matrix-valued Nevanlinna-Pick
      problem for this data set is
      \begin{equation}\label{pick-mat}
{\mathbb P} = \left[ \frac{I_{k} - W_{i} W_{j}^{*}}
      {1 - z_{i} \overline{z}_{j}} \right]_{i,j = 1, \dots n}.
\end{equation}
We define auxiliary matrices
\begin{equation}\label{ZEW}
Z = \begin{bmatrix} z_{1}I_{k} & & \\ & \ddots & \\ & &
          z_{n} I_{k} \end{bmatrix}, \quad E = \begin{bmatrix} I_{k} \\
          \vdots \\ I_{k} \end{bmatrix}, \quad
           W = \begin{bmatrix} W_{1} \\ \vdots \\ W_{n} \end{bmatrix}.
\end{equation}
      For $X$ a free-parameter $k \times k$
      matrix, known results on matrix-valued Carath\'eodory-Fej\'er
      interpolation (see e.g.~\cite{BGR}) show that the
       Pick matrix for the  Carath\'eodory-Fej\'er interpolation problem
      of finding $S \in (\cS)^{k \times k}$ satisfying the
      extended set of interpolation conditions
      $$
      S(z_{j}) = W_{j}\quad \text{for}\quad j = 1, \dots, n \quad\text{and}
      \quad S(0) = X, \; \; S^\prime(0) = 0
      $$
      is given by
      \begin{equation}  \label{Pick-X'}
       {\mathbb P}_{X}' = \begin{bmatrix}
       {\mathbb P}&E - WX^{*} & Z (E - WX^{*})\\
       E^{*} - X W^{*} & I_{k}-XX^* & 0\\
       (E^{*} - X W^{*})Z^{*} & 0 & I_{k}-XX^*
       \end{bmatrix}.
      \end{equation}
A standard Schur-complement manipulation shows that the Pick matrix ${\mathbb P}_{X}'$
being positive-semidefinite is equivalent to the
positive-semidefiniteness of the matrix
\begin{equation}  \label{Pick-X}
       {\mathbb P}_{X} = \begin{bmatrix}
       {\mathbb P}&E - WX^{*} & Z (E - WX^{*})&0&0\\
       E^{*} - X W^{*} & I_{k}&0&X&0\\
       (E^{*} - X W^{*})Z^{*} & 0 & I_{k}&0&X\\
       0&X^*&0&I_k&0\\
       0&0&X^*&0&I_k
       \end{bmatrix}.
      \end{equation}
      We are thus led to the following result.

      \begin{theorem}  \label{T:2}
         The problem {\bf MCNP} has a solution if and only if there
          exists a $k \times k$ matrix $X$ so that the matrix ${\mathbb
          P}_{X}$ given by \eqref{Pick-X} is positive-semidefinite.
       \end{theorem}

       For the case $k=1$ the criterion in Theorem \ref{T:2}, while of
       the same flavor, is somewhat different from the criterion in
       Theorem \ref{T:2DPRS} since the free parameter $X = [x]$ in the
       matrix ${\mathbb P}_{x}$ appears linearly in \eqref{Pick-X} while
       the free parameter $\lambda$ in the matrix ${\mathbb P}_{\lambda}$
in
       \eqref{clas-crit2} appears in a linear-fractional form in the
       matrix entries. There is an advantage in the linear
       representation \eqref{Pick-X} over the fractional representation
       \eqref{clas-crit2} (or the quadratic in \eqref{Pick-X'}); specifically, the search for a parameter $X$
       which makes the matrix ${\mathbb P}_{X}$ positive-semidefinite is
       a problem of a standard type called a Linear Matrix Inequality
       (LMI) for which efficient numerical algorithms are available for
       determining if a solution exists (see \cite{NN}).
         However in Section
       \ref{S:LMI} we also give an analytic treatment concerning
       existence and parametrization of solutions for the LMIs
       occurring here; only in some special cases does one get a clean
       positivity test for existence of solutions and, when solutions
       exist, a complete linear-fractional description  for the set
       of all solutions. The type of LMIs occurring here are similar
       to the {\em structured LMIs} arising in the robust control
       theory for control systems subject to structured balls of
       uncertainty with a so-called linear-fractional-transformation
       model---see \cite{DP}.

       It is not immediately clear how to convert the criterion in
Theorem \ref{T:1}
       to an LMI; what is true is that if one is lucky enough to find a
       particular $(\alpha, \beta) \in {\mathbb G}(1 \times 1)$ along with
a
       tuple $X_{1}, \dots, X_{n}$ of $k \times \ell$ matrices for which
       the matrix \eqref{op-ker-family} is not positive-semidefinite,
       then it necessarily follows that the matrix-valued constrained
       Nevanlinna-Pick interpolation problem does not have a solution.

       We remark that Theorems \ref{T:1DPRS} and \ref{T:1} can be seen as
       parallel to the situation for Nevanlinna-Pick interpolation over
       a finitely-connected planar domain (see \cite{Abrahamse, Ball79}).
       For $\cR$ a bounded domain in the complex plane bounded by $g+1$
       disjoint analytic Jordan curves, Abrahamse \cite{Abrahamse}
identified a family of
       reproducing kernel Hilbert spaces $H^{2}_{\bu}(\cR)$ indexed by
       points $\bu$ in the $g$-torus ${\mathbb T}^{g}$ with
       corresponding kernels $K_{\bu}(z,w)$.  His result is that, given
        $n$ distinct points $z_{1}, \dots, z_{n}$ in $\cR$ and $n$
        complex values $w_{1}, \dots, w_{n}$, there exists a
        holomorphic function $f$ mapping $\cR$ into the closed unit disk
        $\overline{\mathbb D}$ which also satisfies the interpolation
        conditions $f(z_{i}) = w_{i}$ for $1 \le i \le n$ if and only if
        the $n \times n$ matrix $[(1 - w_{i} \overline{w_{j}})
        K_{\bu}(z_{i}, z_{j})]_{i,j=1, \dots, n}$ is positive
        semidefinite for all $\bu \in {\mathbb T}^{g}$. The first author
        \cite{Ball79} obtained a commutant lifting theorem for the
        finitely-connected planar-domain setting which, when specified
        to a matrix-valued interpolation problem over $\cR$, yields the
        following result.

        \begin{theorem} \label{T:Ball79}  For each $\ell = 1,2, \dots$
        there exists a family of $\ell
       \times \ell$ matrix-valued kernels $K_{\bU}(z,w)$ on $\cR$
       indexed by $g$-tuples $\bU = (U_{1}, \dots, U_{g}) \in
       \cU(\ell)^{g}$ of $\ell \times \ell$ unitary matrices so that the
       following holds:  given distinct points $z_{1}, \dots, z_{n}
       \in \cR$ and $k \times k$ matrices $W_{1}, \dots, W_{n}$,
       there exists a holomorphic function $F$ on $\cR$ with
       values in the closed unit ball of $k \times k$ matrices
       which satisfies the interpolation conditions
       \begin{equation}   \label{R-int}
         F(z_{i}) = W_{i} \quad \text{for}\quad i = 1, \dots, n
       \end{equation}
       if and only if, for any choice of $\bU \in \cU(\ell)^{g}$
       and $n$-tuple $X_{1}, \dots, X_{n}$ of $k \times \ell$
       matrices for $1 \le \ell \le k$, it holds that
       \begin{equation}   \label{R-matrixPick}
       \sum_{i,j = 1, \dots, n} \operatorname{trace} \left[ X_{j}
       K_{\bU}(z_{i}, z_{j}) X_{i}^{*} - W_{j}^{*} X_{j}
       K_{\bU}(z_{i}, z_{j}) X_{i}^{*} W_{i} \right] \ge 0.
       \end{equation}
     \end{theorem}

     We conclude that Theorem \ref{T:1DPRS} can be viewed as an analogue
     of Abrahamse's result \cite{Abrahamse} while Theorem \ref{T:1} can
     be viewed as the parallel analogue of the matrix-valued extension
     of Abrahamse's result Theorem \ref{T:Ball79}.

     In the sequel we actually formulate and prove a more general version
    of Theorem \ref{T:2}, where the algebra $H^{\infty}_{1}$ is replaced
    by an algebras of the form $H^{\infty}_{B}: =
    {\mathbb C} + B H^{\infty}$ with $B$ a finite Blaschke product.
    The algebra $H^{\infty}_{1}$ is recovered as the special case where
    $B(z) = z^{2}$.  Particular examples of the case where each zero of
    $B(z)$ has simple multiplicity were studied in \cite{Solazzo} and
other
    results in this direction are given in \cite{Raghu}.
         Additional motivation for the study of algebras like
$H^{\infty}_{B}$
    comes from the theory of algebraic curves:
         Agler-McCarthy \cite{AMcC} have shown that the algebra
    $H^{\infty}_{1}$ is isometrically isomorphic to the algebra
    $H^{\infty}(V)$ of bounded analytic functions on the variety
    $\{(z^{2}, z^{3}) \colon z \in {\mathbb D}\}$; we expect that
    algebras of the type $H^{\infty}_{B}$ will play a similar role
    for more general varieties.
     We expect that Theorem \ref{T:1} can also be extended to the
    setting where the algebra $H^{\infty}_{1}$ is replaced by
$H^{\infty}_{B}$; we
    have chosen to restrict ourselves to the case $B(z) = z^{2}$ to
    keep the notation simple and explicit.

     The paper is organized as follows.  In Section \ref{S:T1} we prove
      Theorem \ref{T:1}. As a corollary we obtain a distance formula for
     the subspace $\cI_{{\mathfrak D}}$ of elements in
$(H^{\infty}_{1})^{k \times k}$
     satisfying the associated set of homogeneous interpolation
     conditions;  this extends another scalar-valued result from \cite{DPRS}.
     We also obtain a Beurling-Lax theorem
     for this setting and provide some detail on how the lifting
     theorem from \cite{Ball79} leads to Theorem \ref{T:Ball79}.
In Section \ref{S:T2} we prove our extension of Theorem
     \ref{T:2}.  In Section \ref{S:LMI} we provide a
     further analysis of the LMIs which arise here.
     We use this analysis in Section \ref{S:bodies} to discuss the
     geometry of the set of values $w = s(z_{0})$ (where $z_{0}$ is a
     nonzero point in the disk not equal to one of the interpolation
     nodes $z_{1}, \dots, z_{n}$)  associated with the
     set of all constrained Schur-class solutions  of a constrained
     Nevanlinna-Pick interpolation problem.

     \section{Proof of Theorem \ref{T:1}}  \label{S:T1}

     \subsection{Preliminaries}

     We review some preliminaries and notation.  $\cC^{2, k \times
     \ell}$ denotes the space of $k \times \ell$ complex matrices
     considered as a Hilbert space with the inner product
     $$
      \langle X, Y \rangle_{\cC^{2, k \times \ell}} =
      \operatorname{trace} [ Y^{*} X ].
     $$
     Similarly, $(L^{2})^{k \times \ell}$ is the space of $k \times
     \ell$ matrices with entries equal to $L^{2}$ functions on the
     circle ${\mathbb T}=\{\lambda\in{\mathbb C}\colon |\lambda|=1\}$ considered
     as a Hilbert space with the inner product
     $$
       \langle F, G \rangle_{(L^{2})^{k \times \ell}} = \frac{1}{2 \pi}
       \int_{{\mathbb T}} \operatorname{trace}[G(z)^{*} F(z)]\, |dz|.
      $$
      Related spaces are the Banach space $(L^{1})^{k \times \ell}$
      of $k \times\ell$ matrices with entries equal to $L^{1}$ functions on ${\mathbb T}$
      with norm given by
      $$
       \| F \|_{(L^{1})^{k \times \ell}} = \frac{1}{2 \pi}
       \int_{{\mathbb T}} \operatorname{trace} [ (F(z)^{*}
       F(z))^{1/2}]\, |dz|,
      $$
      and the Banach space $(L^{\infty})^{k \times \ell}$ of $k
\times\ell$
      matrices with entries equal to $L^{\infty}$ functions on ${\mathbb T}$ with the
      supremum norm $\|\ \cdot  \|_\infty$. We will also use the Hilbert
space
      $(H^{2})^{k \times \ell}$ and the Banach spaces $(H^{1})^{k \times
\ell}$
      and $(H^{\infty})^{k \times \ell}$ which are the restrictions
      of $(L^{2})^{k \times \ell}$, $(L^{1})^{k \times \ell}$
      and $(L^{\infty})^{k \times \ell}$, respectively, to the matrices
whose
      entries are analytic functions.

      The Riesz representation theorem for this setting
      identifies the dual of $(L^{1})^{k \times \ell}$ as the space
      $(L^{\infty})^{\ell \times k}$ via the pairing
      \begin{equation}  \label{pairing}
      [ F, G ]_{(L^{\infty})^{\ell \times k} \times (L^{1})^{k \times
      \ell}} = \frac{1}{2 \pi} \int_{{\mathbb T}} \operatorname{trace}
      [ F(z) G(z) ]\, |dz|.
      \end{equation}
      Moreover, the pre-annihilator of $(H^{\infty})^{k \times \ell}$
under this pairing is given by
      \begin{equation} \label{Hinf-preann}
     ((H^{\infty})^{k \times  \ell})_{\perp} = z \cdot
(H^{1})^{\ell \times k}.
    \end{equation}

    We shall make use of the following matrix-valued analogue of the
    Riesz factorization theorem.

    \begin{theorem} \label{T:Riesz}
        Any $h \in (H^{1})^{k \times k}$ can by factored as
        $$
        h = g f_{0}
        $$
        where $g \in (H^{2})^{k \times \ell}$ is such that its
        transpose $g^{\top}$ is outer, $f_{0} \in (H^{2})^{\ell \times
k}$,
        and
        $$
        \|f_{0}\|_{2}^{1/2} = \| g \|_{2}^{1/2} = \| h\|_{1}.
        $$
        \end{theorem}

\begin{proof}
    It is well known (see \cite{NF}) that $h^{\top}$
       has an inner-outer factorization $h^{\top} = h_{i}^{\top}
       h_{o}^{\top}$ for a ${k \times \ell}$ inner function
       $h_{i}$ and an outer function $h_{o}^{\top}$ in
       $(H^{1})^{\ell \times k}$.  We next factor $h_{o}$ as
       $$ h_{o} = g \widetilde f_{0}
       $$
       where $g \in (H^{2})^{k \times \ell}$, $\widetilde f_{0}
       \in (H^{2})^{\ell \times k}$ and
       $$
       \widetilde f_{0}^{*} \widetilde f_{0} = (h_{o}^{*} h_{o})^{1/2}, \quad
       g^{*}g = \widetilde f_{0} \widetilde f_{0}^{*} \text{ on }
       {\mathbb T}
       $$
       (see \cite[Theorem 4]{Sarason} where the operator-valued
       version is done).
       Then $h = g f_{0}$ with $f_{0} = \widetilde f_{0} h_{i}$
        has the requisite properties.
    \end{proof}

    \subsection{Representation spaces for $(H^{\infty}_{1})^{k \times
    k}$ and proof of necessity}

     We fix a positive integer $k$.  Let $1 \le \ell  \le \ell' \le k$.
     We define ${\mathbb G}(\ell' \times \ell)$ as in the introduction:
     \begin{equation}  \label{G(ell)}
         {\mathbb G}(\ell' \times \ell) = \{ (\alpha, \beta) \colon
\alpha, \beta \in
         {\mathbb C}^{\ell' \times \ell},\ \  \alpha \alpha^{*} +
         \beta \beta^{*} = I_{\ell'} \text{ and } \operatorname{ker}
         \alpha =\{0\}\}.
     \end{equation}
     For $(\alpha, \beta) \in {\mathbb G}(\ell' \times \ell)$ let
     $H^{2}_{\alpha, \beta}$ be the subspace
     \begin{equation}  \label{H2ell}
     H^{2}_{\alpha, \beta} = \{ F \in (H^{2})^{k \times \ell}
     \colon \text{ for some } U \in {\mathbb C}^{k \times \ell'}, \,
      F(0) = U \alpha \text{ and } F'(0) = U \beta\}
     \end{equation}
     of $(H^{2})^{k \times \ell}$. Then $H^{2}_{\alpha, \beta}$ is a
     Hilbert space in its own right with norm taken to be the restriction
      of the norm of $(H^{2})^{k \times \ell}$.
     These spaces turn out to be models for representations of the algebra
     $(H^{\infty}_{1})^{k \times k}$.

     \begin{proposition}  \label{P:1}
     Let $(\alpha, \beta) \in {\mathbb G}(\ell' \times \ell)$. If $S \in
      (H^{\infty}_{1})^{k \times k}$ and $f \in H^{2}_{\alpha, \beta}$,
      then also $S\cdot f \in H^{2}_{\alpha, \beta}$.
     Moreover, the multiplication operator $M_S$ of $S$ on
     $H^{2}_{\alpha, \beta}$ satisfies
      \begin{equation}  \label{mult-norm}
       \| M_{S}\|_{op} = \| S \|_{\infty}.
      \end{equation}
      \end{proposition}

      \begin{proof}
    If $S(z) = S_{0} + z^{2} \widetilde S(z)
     \in (H^{\infty}_{1})^{k \times k}$ with $\widetilde S \in
     (H^{\infty})^{k \times k}$, then
     $$
     S_{0} + z^{2} \widetilde S(z)) (U \alpha + z U \beta +
     O(z^{2})) = (S_{0} U) \alpha + z(S_{0}U) \beta + O(z^{2}) \in
     H^{2}_{\alpha, \beta}
     $$
     for any $U \alpha + z U \beta + O(z^{2}) \in H^{2}_{\alpha,
     \beta}$.

     We temporarily write $\widetilde M_S$ for the multiplication operator
     of $S$ on $(H^2)^{k\times\ell}$. For $\ell=1$ it is well known that $\widetilde M_S$
     satisfies $\|\widetilde M_S\|_{op}=\|S\|_\infty$ (see e.g.~\cite{NF}), and it is not
     difficult to see that this equality holds for the case $\ell\not=1$ as well. We write
     $M_S$ for the multiplication operator of $S$ as an
     operator on $H^{2}_{\alpha, \beta}$. Since
     the norm of $H^{2}_{\alpha, \beta}$ is just the restriction of the
     norm of $(H^2)^{k\times\ell}$, the inequality $\|M_S\|_{op}\leq\|S\|_\infty$ is
     immediate. Choose $f_n\in (H^2)^{k\times\ell}$ of norm 1 so that
     $\|\widetilde M_S f_n\|\to\|S\|_\infty$. Then the functions $g_n(z):=z^2f_n(z)$
     belong to $H^{2}_{\alpha, \beta}$, and $\|M_Sg_n\|_{H^{2}_{\alpha, \beta}}
     =\|\widetilde M_S f_n\|_{(H^2)^{k\times\ell}}\to\|S\|_\infty$,
     which proves that $\|M_S\|_{op}= \|S\|_\infty$.
\end{proof}

    \begin{remark} {\em In case $\alpha$ is invertible, it can be shown
        that $(H^{\infty}_{1})^{k \times k}$ is {\em exactly} the left
        multiplier space for $H^{2}_{\alpha, \beta}$ (see
        \cite[Proposition 3.1]{DPRS} for the case $k=1$).  In the proof
        of Theorem \ref{T:1} to follow, what comes up is the case
        $\alpha$ injective.  As the reproducing kernels $K^{\alpha, \beta}$
        in (\ref{RK}) can be approximated by reproducing
        kernels $K^{\alpha', \beta'}$ with $\alpha'$
        injective, Theorem \ref{T:1} remains valid if one restricts to
        $(\alpha, \beta)$ with $\alpha$ injective.
        } \end{remark}

        \begin{proposition} \label{P:2}
        For $(\alpha, \beta) \in {\mathbb G}(\ell' \times \ell)$, the
space
       $H^{2}_{\alpha, \beta}$ is a reproducing kernel
       Hilbert space with reproducing kernel $K^{\alpha, \beta}$ given
      by \eqref{RK} and  having the reproducing property
      \begin{equation}  \label{RKprop}
      \langle f(w), X \rangle_{\cC^{2, k \times \ell}} =
      \langle f, X K^{\alpha, \beta}(\cdot, w)
      \rangle_{H^{2}_{\alpha, \beta}}.
      \end{equation}
      Moreover, for $F \in (H^{\infty}_{1})^{k \times k}$, the operator
$M_{F} \colon
      f \mapsto F \cdot f$ of multiplication on the left by $F$ takes
      $H^{2}_{\alpha, \beta}$ into itself and the action of
      $M_{F}^{*}$ on kernel elements $X K^{\ell, (\alpha, \beta)}(w,
      \cdot)$ is given by
      \begin{equation}  \label{MF*kernel}
     M_{F}^{*} X K^{\alpha, \beta}(\cdot, w) = F(w)^{*} X
     K^{\alpha, \beta}(\cdot, w).
      \end{equation}
      \end{proposition}

      \begin{proof}  Let $f(z) = U \alpha + z U \beta + z^{2} \widetilde
     f(z)$ be a generic element of $H^{2}_{\alpha, \beta}$.
     By direct computation we have, for $X \in {\mathbb C}^{k
     \times \ell}$,
     \begin{align*}
       &  \langle f, X K^{\alpha, \beta}(\cdot,w)
         \rangle_{H^{2}_{\alpha, \beta}}  =
         \langle U\alpha + z U \beta + z^{2} \widetilde f(z), X
         K^{\alpha, \beta}(\cdot, w)
         \rangle_{H^{2}_{\alpha, \beta}} \\
         &\qquad  = \langle U \alpha + z U \beta, X (\alpha^{*} +
         \overline{w} \beta^{*}) (\alpha + z \beta)
         \rangle_{H^{2}_{\alpha, \beta}}  +
         \left\langle z^{2} \widetilde f(z), X \frac{ \overline{w}^{2}
         z^{2}}{1 - \overline{w} z} \right\rangle_{H^{2}_{\alpha,
         \beta}}  \\
         & \qquad =\operatorname{trace}\left[ (\alpha + w \beta)
         X^{*} U (\alpha \alpha^{*} + \beta \beta^{*}) \right] +
         \langle w^{2} \widetilde f(w), X \rangle_{\cC^{2, k \times
         \ell}} \\
         & \qquad = \langle U(\alpha + w \beta), X
         \rangle_{\cC^{2,k \times \ell}} + \langle w^{2} \widetilde
         f(w), X \rangle_{\cC^{2,k \times \ell}} \\
         & \qquad = \langle f(w), X \rangle_{\cC^{2, k \times \ell}}
      \end{align*}
      in agreement with \eqref{RKprop} as wanted. We next compute, for $f
\in H^{2}_{\alpha, \beta}$, $F \in
      (H^{\infty}_{1})^{k \times k}$ and $X \in {\mathbb C}^{k \times
      \ell}$,
      \begin{align*}
      \langle f, M_{F}^{*} X K^{\alpha, \beta}( \cdot, w)
     \rangle_{H^{2}_{\alpha, \beta}}& =
     \langle M_{F} f, X K^{\alpha, \beta}( \cdot, w)
     \rangle_{H^{2}_{\alpha, \beta}} \\
     &  = \langle F(w) f(w), X \rangle_{\cC^{2, k \times
     \ell}} \\
     & = \langle f(w), F(w)^{*} X \rangle_{\cC^{2, k \times
     \ell}} \\
     & = \langle f, F(w)^{*} X K^{\alpha, \beta}(\cdot, w)
     \rangle_{H^{2}_{\alpha, \beta}}
        \end{align*}
        and thus \eqref{MF*kernel} follows as well.
        \end{proof}

\begin{proof}[Proof of necessity in Theorem \ref{T:1}]
       Suppose that there exists a function $S$ in the constrained
       Schur class
       $(\cS_{1})^{k \times k}$ satisfying interpolation conditions
       \eqref{op-int} for given $n$ distinct nonzero points $z_{1},
       \dots, z_{n}$ in ${\mathbb D}$ and $k \times k$
       matrices $W_{1}, \dots, W_{n}$. Given
      $(\alpha, \beta)\in{\mathbb G}(\ell' \times \ell)$
     for some $\ell',\ell$
     ($1 \le \ell \le \ell' \le k$), define a subspace $\cM$ of $H^{2}_{
     \alpha, \beta}$ by
     $$
     \cM = \overline{\operatorname{span}} \left\{ X K^{
     \alpha, \beta}(\cdot, z_{j}) \colon X \in \cC^{2, k \times
     \ell}, \, j = 1, \dots, n \right\}.
     $$
     {}From \eqref{MF*kernel} we see that $\cM$ is invariant under
     $M_{S}^{*}$.  {}From the assumption that $\| S \|_{\infty} \le
     1$, we know by \eqref{mult-norm} that $\left\|
     M_{F}^{*}|_{\cM} \right\| \le 1$.  Hence, for a generic
     element
     \begin{equation}  \label{f}
         f = \sum_{j=1}^{n} X_{j} K^{\alpha, \beta}(\cdot, z_{j}) \in \cM,
      \end{equation}
      necessarily
      $$
        \|f\|^{2} - \| M_{S}^{*} f \|^{2} \ge 0.
      $$
      Substituting \eqref{f} for $f$ into the latter inequality, expanding
out inner products and using the reproducing property \eqref{RKprop} and
the equality (\ref{MF*kernel}) then leaves us with
\eqref{op-ker-family}.
\end{proof}

     \subsection{Proof of sufficiency} The proof will follow
      the duality proof as in \cite{Sarason, DPRS} using
     the adaptations for the matrix-valued case as done in
     \cite{Ball79} in a different context.
          A key ideal in the
     algebra $(H^{\infty}_{1})^{k \times k}$ is the set
     of all functions $F \in (H^{\infty})^{k \times k}$
     which satisfy the homogeneous interpolation conditions
     associated with the data set ${\mathfrak D}$:
     \begin{equation}  \label{cID}
     \cI_{{\mathfrak D}} :
     = \{ F \in (H^{\infty}_{1})^{k \times
     k} \colon F(z_{i}) = 0 \; \; \text{for}\; \;  i=1, \dots, n\}.
     \end{equation}
     The first step is
     to compute the pre-annihilator of $\cI_{{\mathfrak D}}$.
     To this end we introduce the dual version of ${\mathbb
     G}(\ell' \times \ell)$:
     $$
     {\mathbb G}(\ell' \times \ell)^{*} = \{ (a,b) \colon a,b \in {\mathbb
     C}^{\ell \times \ell'}, \, a \text{ onto}, \, a^{*}a +
     b^{*} b  = I_{\ell'}\}.
     $$
     For $(a,b)\in{\mathbb G}(\ell' \times \ell)^{*}$
     we then define associated spaces
     \begin{align*}
     H^{1}_{r,(a,b)}&  = \{ h \in (H^{1})^{ \ell \times k} \colon
     \text{for some } U \in {\mathbb C}^{\ell \times k}, \,
     h(0) = aU \text{ and } h'(0) = bU\}, \\
     H^{\infty}_{r, (a,b)} & = \{ F \in (H^{\infty})^{ \ell \times k}
\colon
     \text{for some } U \in {\mathbb C}^{\ell \times k}, \,
     F(0) = aU \text{ and } F'(0) = bU\}, \\
     H^{2}_{r,(a,b)} & = \{ f \in (H^{2})^{ \ell \times k} \colon
     \text{for some } U \in {\mathbb C}^{\ell \times k}, \,
     f(0) = aU \text{ and } f'(0) = bU\}.
     \end{align*}

     \begin{proposition}  \label{P:3}
         There exists a pair $(a,b) \in {\mathbb G}(k \times k)^{*}$ so
      that the subspace $\cI_{{\mathfrak D}}$
      given by \eqref{cID} can be expressed as
       $$
       \cI_{{\mathfrak D}} = B_{{\mathfrak D}} H^{\infty}_{r,
(a,b)}\quad\mbox{where}\quad
     B_{{\mathfrak D}}(z) = \prod_{k=1}^{n} \frac{z-z_{k}}{1 -
     \overline{z_{k}} z}  I_{k}
     $$
is the scalar Blaschke product with zeros equal to the
interpolation nodes $z_{1}, \dots,
     z_{n}$ given by the data set \eqref{op-data}.
     \end{proposition}
     \begin{proof}
        By definition \eqref{cID} we have $\cI_{{\mathfrak D}} =
        (H^{\infty}_{1})^{k \times k} \cap B_{{\mathfrak D}}
        (H^{\infty})^{k \times k}$.
        Then $F \in \cI_{{\mathfrak D}}$ means that $F$ has the
        form $F = B_{{\mathfrak D}} G $ for a $G \in
        (H^{\infty})^{k \times k}$ with the additional property
        that
        $$
        F'(0) = B_{{\mathfrak D}}(0) G'(0) + B_{{\mathfrak D}}'(0)
        G(0) = 0.
        $$
        By assumption no $z_i$ is equal to $0$ and hence
        $B_{{\mathfrak D}}(0)$ is invertible.  Thus
        we can solve for $G'(0)$ in terms of $G(0)$:
        $$
        G'(0) = - B_{{\mathfrak D}}(0)^{-1} B_{{\mathfrak D}}'(0)
        G(0)
        $$
        or
        $$
        G(0) = \widetilde a \widetilde U, \quad G'(0) = \widetilde
        b \widetilde U
        $$
        where
        $$
        \widetilde a = I_{k}, \quad \widetilde b = -B_{{\mathfrak
        D}}(0)^{-1} B_{{\mathfrak D}}'(0), \quad \widetilde U =
        G(0).
        $$
        If we normalize $(\widetilde a, \widetilde b)$ to $(a,b)$
        according to the formula
       \begin{align}
           & a = \widetilde a (I + B_{{\mathfrak D}}'(0)^{*}
           B_{{\mathfrak D}}(0)^{*-1} B_{{\mathfrak D}}(0)^{-1}
           B_{{\mathfrak D}}'(0))^{-1/2}, \notag \\
           & b = \widetilde b (I + B_{{\mathfrak D}}'(0)^{*}
           B_{{\mathfrak D}}(0)^{*-1} B_{{\mathfrak D}}(0)^{-1}
           B_{{\mathfrak D}}'(0))^{-1/2}, \notag \\
        & U = (I + B_{{\mathfrak D}}'(0)^{*}
            B_{{\mathfrak D}}(0)^{*-1} B_{{\mathfrak D}}(0)^{-1}
           B_{{\mathfrak D}}'(0))^{1/2}\widetilde U,
           \label{cID-ab}
          \end{align}
        then $(a,b) \in {\mathbb G}(k\times k)^{*}$ and
        $$
          G(0) = aU, \quad G'(0)= b U
         $$
         shows that $G \in H^{\infty}_{r, (a,b)}$. Note that the
         $a$ and $b$ constructed in (\ref{cID-ab}) are independent
         of the function $F\in\cI_{{\mathfrak D}}$, so that this
         choice of $a$ and $b$ works for any
         $F\in\cI_{{\mathfrak D}}$.
      \end{proof}

       \begin{remark} \label{R:cID}  {\em For our setting here
           $B_{{\mathfrak D}}(z)$ is a scalar Blaschke product so
           the $(a,b)$ produced by the recipe \eqref{cID-ab} are
           actually scalar $k \times k$ matrices.  The proof of
           the lemma applies more generally to the setting where
           $B_{{\mathfrak D}}$ is not scalar, but we still assume
           that $B_{{\mathfrak D}}(0)$ is invertible.  Due to this
           observation, the analysis here applies equally well
           should we wish to study left-tangential interpolation
           problems in the class $(H^{\infty}_{1})^{k \times k}$
           rather than just full-matrix-valued interpolation as in
           \eqref{op-int}.
           } \end{remark}

       We next compute a pre-annihilator of $\cI_{{\mathfrak D}}$. Here we use the notation
       $\ker_{\ell} X$ and $\operatorname{im}_{\ell} X$ for a matrix
       $X$ to indicate the subspace of row vectors arising as the
       kernel (respectively image) of $X$ when $X$ is considered as an
       operator with   row-vector argument acting on the left, i.e.,
       for $X \in {\mathbb C}^{k \times k'}$,
       \begin{align*}
     &  \ker_{\ell} X = \{ v \in {\mathbb C}^{1 \times k} \colon v
       X = 0\}, \\
     &  \operatorname{im}_{\ell} X = \{ u \in {\mathbb C}^{1
       \times k'} \colon u = v X \text{ for some } v \in {\mathbb
       C}^{1 \times k}\}.
    \end{align*}

       \begin{proposition}  \label{P:4} Suppose that
           $\cI_{{\mathfrak D}} = (H^{\infty}_{1})^{k \times k}
           \cap B_{{\mathfrak D}} (H^{\infty})^{k \times k}$ is
           expressed as $B_{{\mathfrak D}} H^{\infty}_{r, (a,b)}$ as in
           Proposition \ref{P:3}.  Then the pre-annihilator
           $(\cI_{{\mathfrak D}})_{\perp}$ of $\cI_{{\mathfrak
           D}}$ is given by
           \begin{equation}  \label{cIDperp}
           (\cI_{{\mathfrak D}})_{\perp} = z^{-1} H^{1}_{
           a_{\perp}, b_{\perp}}B_{{\mathfrak D}}^{*}
       \end{equation}
       where $(a_{\perp}, b_{\perp}) \in {\mathbb G}(k \times k)$ is
chosen
       so that
       \begin{equation}  \label{aperpbperp}
       \operatorname{im}_{\ell}  \begin{bmatrix} b_{\perp} &
       a_{\perp}
       \end{bmatrix} = \operatorname{ker}_{\ell} \begin{bmatrix} a
       \\ b \end{bmatrix}.
       \end{equation}
       \end{proposition}

       \begin{proof}
           Let $F(z) = B_{{\mathfrak D}}(z) (a U + z bU +z^{2}
           \widetilde F(z))$ be a generic element of
           $\cI_{{\mathfrak D}}$ and suppose that $h \in
           (L^{1})^{k \times k}$ is in the pre-annihilator
           $(\cI_{{\mathfrak D}})_{\perp}$.  Then
          $$
          [F,h]_{(L^{\infty})^{k \times k} \times (L^{1})^{k
          \times k}} = \frac{1}{2 \pi} \int_{{\mathbb T}}
          \operatorname{trace}\left[ B_{{\mathfrak D}}(z) (aU +
          zbU + z^{2} \widetilde F(z)) h(z) \right]\, |dz| = 0
          $$
   for all such $F$.  In particular
   $$
   \frac{1}{2 \pi} \int_{{\mathbb T}} \operatorname{trace} \left[
   B_{{\mathfrak D}}(z) z^{2} \widetilde F(z) h(z) \right]\, |dz| = 0
   $$
   for all $\widetilde F \in (H^{\infty})^{k \times k}$ from which we
   conclude that $z^{2} h(z) B_{{\mathfrak D}}(z) \in \left(
   (H^{\infty})^{k \times k} \right)_{\perp} = z(H^{1})^{k \times k}$.
   We may therefore write $h(z) = z^{-1} \widetilde h(z) B_{{\mathfrak
   D}}(z)^{*}$ where $\widetilde h(z) \in (H^{1})^{k \times k}$.  Write
   $$
   \widetilde h(z) = \widetilde h(0) + z \widetilde h'(0) + z^{2}
   \widetilde{\widetilde h}(z)\quad \text{with}\quad \widetilde{\widetilde
   h}\in (H^{1})^{k \times k}.
   $$
   Then $h \in (\cI_{{\mathfrak D}})_{\perp}$ forces in addition
   \begin{align*}
       0 & =     \left[B_{{\mathfrak D}}(z) (aU + z bU), \,
       (z^{-1}\widetilde h(0) + \widetilde h'(0) + z
       \widetilde{\widetilde h}(z)) B_{{\mathfrak D}}(z)^{*}
       \right]_{(L^{\infty})^{k \times k} \times (L^{1})^{k \times k}} \\
    & = \frac{1}{2 \pi} \int_{{\mathbb T}} \operatorname{trace} \left[
(aU + zbU) (z^{-1}
    \widetilde h(0) + \widetilde h'(0)) \right] \, |dz| \\
    & =  \operatorname{trace} \left[ aU \widetilde h'(0) + b U
    \widetilde h(0) \right] =
     \operatorname{trace} \left[ U
    (\widetilde h'(0) a + \widetilde h(0) b) \right]
    \end{align*}
    for all $k \times k$ matrices $U$.  As the analysis is reversible,
    we conclude that $h \in (\cI_{{\mathfrak D}})_{\perp}$ if and only
    if $h(z) = z^{-1} \widetilde h(z) B_{{\mathfrak D}}(z)^{*}$ where
    $\widetilde h(z) \in (H^{1})^{k \times k}$ is such that
    $$
      \begin{bmatrix} \widetilde h'(0)  & \widetilde h(0) \end{bmatrix}
     \begin{bmatrix} a \\ b \end{bmatrix} = 0.
    $$
    {}From the definitions, this is just the assertion that $h \in z^{-1}
    H^{1}_{a_{\perp}, b_{\perp}}$ where $(a_{\perp},
    b_{\perp})$ is constructed as in \eqref{aperpbperp}. It is also not
    difficult to see that injectivity (and hence invertibility since
    $a$ is square) of $a$ is then equivalent
    to surjectivity (and hence also invertibility) of $a_{\perp}$,
i.e., $(a,b) \in {\mathbb
    G}(k \times k)^{*}$ is equivalent to $(a_{\perp}, b_{\perp}) \in
{\mathbb
    G}(k \times k)$.
    \end{proof}

    \begin{proposition}  \label{P:5}
        Suppose that $h \in (H^{1})^{k \times k}$ is in $(\cI_{{\mathfrak
        D}})_{\perp}$.  Then there is an $(\alpha, \beta) \in
        {\mathbb G}(\ell',\ell)$ for some $1 \le \ell \le \ell' \le k$
        so that  $h$ can be factored in the form
        \begin{equation}  \label{fact}
        h(z) = g(z) f(z)^{*}
        \end{equation}
        where
        \begin{equation}  \label{factor-props}
       g \in H^{2}_{\alpha, \beta}, \quad
        f \in (L^{2})^{k \times \ell} \ominus \left( H^{2}_{\alpha,
        \beta} \cap
        B_{{\mathfrak D}} (H^{2})^{k \times \ell} \right), \quad
        \|g\|_{2}^{1/2}  = \|f\|_{2}^{1/2} = \|h\|_{1}.
        \end{equation}
        \end{proposition}

        \begin{proof}
       By Proposition \ref{P:4} we may write $h = z^{-1}
       \widetilde h B_{{\mathfrak D}}^{*}$ where $\widetilde h
       \in H^{1}_{a_{\perp}, b_{\perp}}$.  As $\widetilde h$ is in
       $(H^{1})^{k \times k}$, by Theorem \ref{T:Riesz}
       we may factor $\widetilde h$ as $\widetilde h = g \cdot f_{0}$,
       where $g \in (H^{2})^{k \times \ell}$, $f \in (H^{2})^{\ell
       \times k}$, $g^{\top}$ outer, and
       \begin{equation}\label{1}
       \|g\|_{2}^{1/2} = \| f_{0}\|_{2}^{1/2} = \| h \|_{1}.
\end{equation}
Combining  factorizations for $h$ and $\widetilde h$ gives
       $h = z^{-1} \widetilde h B_{{\mathfrak D}}^{*} =
       z^{-1} g f_{0}B_{{\mathfrak D}}^{*}  = g f^{*}$
       where we have set
       $$
         f = z B_{{\mathfrak D}} f_{0}^{*}
       $$
       and \eqref{fact} holds with $g \in H^{2}_{\alpha, \beta}$ and $f$
       so constructed. Since $\|f\|_2=\|f_0\|_2$, the norm
       equalities in the third part of \eqref{factor-props} follow from
       \eqref{1}. Since $g^{\top}$ is outer, $g(0)$ has trivial kernel.
        Let
        \begin{equation}\label{eq3}
\ell':={\rm  rank} \, (g(0) g(0)^{*} + g'(0) g'(0)^{*}).
\end{equation}
Assuming that
       $h \ne 0$, we then have $1 \le \ell \le \ell' \le k$.
       Then there exists an injective $k  \times \ell'$ matrix $X$
       which solves the factorization problem

       $$
       g(0) g(0)^{*} + g'(0) g'(0)^{*} = X X^{*}.
       $$
       Define $\ell' \times \ell$ matrices $\alpha,\beta$ by
       $$
       \alpha = (X^{*}X)^{-1} X^{*} g(0), \quad \beta =
       (X^{*}X)^{-1} X^{*} g'(0).
       $$
       Then one can check that $\alpha$ is injective and
       $\alpha \alpha^{*} + \beta \beta^{*} = I_{\ell'}$, i.e.,
       $(\alpha, \beta) \in {\mathbb G}(\ell' \times \ell)$.
       Moreover, (\ref{eq3}) implies that $\textup{Im}\, \begin{bmatrix} g(0) & g'(0)\end{bmatrix}=\textup{Im}\,X$,
       and thus $\begin{bmatrix} g(0) & g'(0)\end{bmatrix} = X \begin{bmatrix} \alpha & \beta \end{bmatrix}$.
       Hence $g \in H^{2}_{\alpha, \beta}$.

       It remains to verify that $f \in
(L^{2})^{k \times k}
       \ominus \left(H^{2}_{a,b} \cap B_{{\mathfrak D}} (H^{2})^{k
       \times \ell}\right)$.
To this end we observe first that
       $(H^{\infty}_{1})^{k  \times k} \cdot g$ is dense in $H^{2}_{a,b}$.
       Indeed, the $L^{2}$-closure of $(H^{\infty}_{1})^{k  \times k} \cdot g$
       satisfies
$$
[(H^{\infty}_{1})^{k \times k} g ]^{-}=
              {\mathbb C}^{k \times k} g +
         [z^{2} (H^{\infty})^{k \times k} g  ]^{-}= {\mathbb C}^{k \times
k} g + z^{2} (H^{2})^{k
              \times k} = H^{2}_{a,b}.
$$
The second identity follows because $g^\top$ is outer.
       Next we observe that
       \begin{equation} \label{claim1}
          [\cI_{{\mathfrak D}} g]^{-} =
           H^{2}_{a,b} \cap B_{{\mathfrak D}} (H^{2})^{k \times
           \ell},
        \end{equation}
        again since $g^{\top}$ is outer. Indeed, the containment
        $\subset$ is clear as the left-hand side is in
        $H^{2}_{a,b}$ and vanishes at the points $z_{1}, \dots,
        z_{d}$. Moreover, evidently both sides of \eqref{claim1}
        have codimension $n$ in $H^{2}_{a,b}$ and \eqref{claim1}
        follows.

        Hence, to check that $f$ is orthogonal to $H^{2}_{a,b}
        \cap B_{{\mathfrak D}} (H^{2})^{k \times k}$, it suffices
        to check that $f$ is orthogonal to $\cI_{{\mathfrak D}}
        g$.  As $\cI_{{\mathfrak D}} = B_{{\mathfrak D}}
        H^{\infty}_{r, (a,b)}$ by  Proposition \ref{P:3}, we
        conclude that it suffices to check that $f$ is orthogonal
        to  elements of the form
        $\phi:= B_{{\mathfrak D}} F g$ with $F \in H^{\infty}_{r,
        (a,b)}$.  We compute
$$
\langle \phi, f \rangle= \frac{1}{2\pi}
            \int_{{\mathbb D}} \operatorname{trace} \left[ B_{{\mathfrak
            D}} F g \cdot z^{-1} f_{0} B_{{\mathfrak
            D}}^{*}\right] \, |dz|= \frac{1}{2 \pi} \int_{{\mathbb D}}
            \operatorname{trace} \left[ z^{-1} \widetilde h F \right]\, |dz|.
$$
            If we expand out $F$ and $\widetilde h$ as
            $$
            F(z) = F(0) + z F'(0) + z^{2} \widetilde F(z), \quad
            \widetilde h(z) = \widetilde h(0) + z \widetilde h'(0) + z^{2} \widetilde{\widetilde h}(z),
            $$
            we see that the zeroth Fourier coefficient of $z^{-1}
            \widetilde h(z) F(z)$ is simply
            $$
            [z^{-1} \widetilde h(z) F(z)]_{0} = \widetilde h(0) F'(0) + \widetilde h'(0) F(0)
            $$
            and hence
            \begin{equation} \label{toshow}
            \frac{1}{2 \pi} \int_{{\mathbb D}}
                   \operatorname{trace} \left[ z^{-1} \widetilde{h} F
\right]\, |dz|
                   = \operatorname{trace} \left[ \widetilde h(0)
                   F'(0) + \widetilde h'(0) F(0) \right].
         \end{equation}
         As $F \in H^{\infty}_{r, (a,b)}$ and $\widetilde h \in
         H^{1}_{a_{\perp}, b_{\perp}}$, we know that
         there are matrices $U$ and $V$ so that
         $$ \begin{bmatrix} \widetilde h'(0) & \widetilde h(0) \end{bmatrix} = U
         \begin{bmatrix} b_{\perp} & a_{\perp} \end{bmatrix},
             \quad \begin{bmatrix} F(0) \\ F'(0) \end{bmatrix}
             =
             \begin{bmatrix} a \\ b \end{bmatrix} V.
         $$
         Hence
$$
\widetilde h(0)F'(0) + \widetilde h'(0) F(0)=
         \begin{bmatrix} \widetilde h'(0) & \widetilde h(0) \end{bmatrix}
             \begin{bmatrix} F(0) \\ F'(0) \end{bmatrix}= U
\begin{bmatrix} b_{\perp} & a_{\perp}
             \end{bmatrix} \begin{bmatrix} a \\ b
             \end{bmatrix} V.
$$
As the left hand side expression equals zero by \eqref{aperpbperp},
it follows that \eqref{toshow} vanishes as needed.
         \end{proof}

\begin{proof}[Proof of sufficiency in Theorem \ref{T:1}]
         We assume that the data set ${\mathfrak D}$ as in
         \eqref{op-data} is such that condition
         \eqref{op-ker-family} is satisfied for all $(\alpha, \beta) \in
     {\mathbb G}(\ell' \times \ell)$ and $n$-tuples $X = (X_{1},
      \dots, X_{n})$ of $k \times \ell$ matrices ($1 \le \ell \le \ell'
       \le k$). By a coordinate-wise application of the theory of Lagrange
         interpolation, we can find a matrix polynomial $Q \in
         (H^{\infty}_{1})^{k \times k}$ which satisfies the
         interpolation conditions \eqref{op-int} (but probably
         {\em not} the additional norm constraint
         $\|Q\|_{\infty} \le 1$).  By the proof of the necessity in
Theorem \ref{T:1} we see that condition
         \eqref{op-ker-family} can be translated to the
         operator-theoretic form
         \begin{equation}  \label{op-form-ker-fam}
             \left\| \left. P_{\cM^{{\mathfrak D}}_{\alpha, \beta}}M_{Q}
             \right|_{\cM^{{\mathfrak D}}_{\alpha, \beta}}
\right\| \le 1 \;
             \text{ for all } \; (\alpha, \beta) \in {\mathbb
             G}(\ell' \times \ell) \; \text{ for } \; 1 \le \ell \leq
             \ell' \le k
         \end{equation}
     where we let $\cM^{{\mathfrak D}}_{\alpha, \beta}$ be the subspace of
     $H^{2}_{\alpha, \beta}$ given by
    $$
    \cM^{{\mathfrak D}}_{\alpha, \beta} = H^{2}_{\alpha, \beta} \ominus
\left(
    H^{2}_{\alpha, \beta} \cap B_{{\mathfrak D}} (H^{2})^{k \times
    \ell} \right).
    $$
         We use $Q$ to define a linear functional on
         $(\cI_{{\mathfrak D}})_{\perp}$ by
         \begin{equation}  \label{linfunc}
             {\mathbb L}_{Q}(h) =
             \frac{1}{2 \pi} \int_{{\mathbb T}}
             \operatorname{trace} \left[ Q(z)  h(z) \right]\,  |dz|.
          \end{equation}
     By Proposition \ref{P:5} we may factor $h$ as $h = g f^{*}$ with
     $g$ and $f$ as in \eqref{factor-props}.  Then we may convert the formula
     for ${\mathbb L}_{Q}(h)$ to an $L^{2}$-inner product as follows:
     \begin{align}
     {\mathbb L}_{Q}(h) & = \frac{1}{2 \pi} \int_{{\mathbb T}}
     \operatorname{trace} \left[ Q(z) h(z) \right]\, |dz| \notag \\
      & = \frac{1}{2 \pi} \int_{{\mathbb T}} \operatorname{trace} \left[
     Q(z) g(z) f(z)^{*} \right] \, |dz| \notag \\
     & = \left\langle Q g, f \right\rangle_{(L^{2})^{k \times \ell}}
    = \left\langle P_{H^{2}_{\alpha, \beta}} Q g, \, P_{\left(
     H^{2}_{\alpha, \beta} \cap B_{{\mathfrak D}}(H^{2})^{k \times
     \ell} \right)^{\perp}} f \right\rangle_{(L^{2})^{k \times \ell}}
     \label{comp1}
     \end{align}
     since $f \in \left(
     H^{2}_{\alpha, \beta} \cap B_{{\mathfrak D}}(H^{2})^{k \times
     \ell} \right)^{\perp}$ by construction and since $H^{2}_{\alpha,
     \beta}$ is invariant under the multiplication operator $M_{Q}$ as
     $Q \in (H^{\infty}_{1})^{k \times k}$. As
     $$
       \cM^{{\mathfrak D}}_{\alpha, \beta} = H^{2}_{\alpha, \beta}  \cap
       \left(
     H^{2}_{\alpha, \beta} \cap B_{{\mathfrak D}}(H^{2})^{k \times
     \ell} \right)^{\perp},
    $$
    we may continue the computation \eqref{comp1} as follows:
    \begin{equation}  \label{comp2}
        {\mathbb L}_{Q}(h)  = \left\langle P_{\cM^{{\mathfrak
D}}_{\alpha, \beta}} Q g,
          P_{\cM^{{\mathfrak D}}_{\alpha, \beta}} f
     \right\rangle_{\cM^{{\mathfrak D}}_{\alpha, \beta}}.
    \end{equation}
    We claim that the linear functional ${\mathbb L}_{Q}$ has
    linear-functional norm at most 1.  To see this we note that
    \begin{align*}
        \| {\mathbb L}_{Q} \| & = \sup_{h \in (\cI_{{\mathfrak
        D}})_{\perp} \colon \|h\|_{1} \le 1} | {\mathbb L}_{Q}(h)| \\
        & = \sup_{ g, \widetilde f \in \cM^{{\mathfrak D}}_{\alpha,
\beta} \colon \|g\|_{2},
        \|\widetilde f\|_{2} \le 1, (\alpha, \beta) \in {\mathbb G}(\ell'
\times
        \ell)}
      \left| \left\langle P_{\cM^{{\mathfrak D}}_{\alpha, \beta}} Q g,
        \widetilde f \right\rangle_{\cM^{{\mathfrak D}}_{\alpha,
\beta}} \right| \le 1
    \end{align*}
    where we use the norm preservation properties \eqref{factor-props} in
    the factorization \eqref{fact} and
    where we use the assumption \eqref{op-form-ker-fam} for the last
    step.  This verifies that $\|{\mathbb L}_{Q} \| \le 1$.

    By the Hahn-Banach Theorem we may extend ${\mathbb L}_{Q}$ to a
    linear functional of norm at most $1$ defined on the whole space
    $(L^{1})^{k \times k}$.  By the Riesz representation theorem for
    this setting (see \eqref{pairing}), such a linear functional has
the form ${\mathbb
    L}_{S}$ for an $S \in (L^{\infty})^{k \times k}$ with
    implementation given by
    $$ {\mathbb L}_{S}(h) = \frac{1}{2 \pi} \int_{{\mathbb T}}
    \operatorname{trace} \left[ S(z) h(z) \right] \, |dz|
    $$
    where also $ \| {\mathbb L}_{S} \| = \| S\|_{\infty}$.
    As $\| {\mathbb L}_{S}\| \le 1$, we have $\| S \|_{\infty}
    \le 1$.  As ${\mathbb L}_{S}$ is an extension of ${\mathbb L}_{Q}$,
    we conclude that
    $$
      {\mathbb L}_{S-Q}|_{(\cI_{{\mathfrak D}})_{\perp}} =
      {\mathbb L}_{S}|_{(\cI_{{\mathfrak D}})_{\perp}} -
      {\mathbb L}_{Q} = 0,
    $$
    or $S-Q \in ((\cI_{{\mathfrak D}})_{\perp})^{\perp} =
    \cI_{{\mathfrak D}}$.  In concrete terms, this means that $S \in
    (H^{\infty}_{1})^{k \times k}$ (since $Q \in (H^{\infty}_{1})^{k
    \times k}$ and $\cI_{{\mathfrak D}}
    \subset H^{\infty}_{1}$), and that $S$ satisfies the
    interpolation conditions \eqref{op-int} (since $Q$ satisfies
    \eqref{op-int} and elements of $\cI_{{\mathfrak D}}$ satisfy the
    associated homogeneous interpolation conditions as in \eqref{cID}).
    As we noted above that $\|S\|_{\infty} \le 1$, we conclude that $S$
    is in fact in $(\cS_{1})^{k \times k}$ and is a solution of the
    interpolation problem as described in
    Theorem \ref{T:1}. This concludes the proof of Theorem \ref{T:1}.
    \end{proof}

      As a corollary we have the following distance formula for
      $\cI_{{\mathfrak D}}\subset(H^{\infty}_{1})^{k \times k}$.

      \begin{corollary}  \label{C:distance}
     Let ${\mathfrak D}$ be a data set as in \eqref{op-data}.
     Then for any $F \in (H^{\infty}_{1})^{k \times k}$,
     \begin{equation}  \label{distance}
     \operatorname{dist}(F, \cI_{{\mathfrak D}}) =
     \sup_{(\alpha, \beta) \in {\mathbb G}(\ell' \times \ell), 1
     \le \ell \le \ell' \le k}  \left\| \left. M_{F}^{*}
     \right|_{\cM^{{\mathfrak D}}_{\alpha, \beta}} \right\|.
     \end{equation}
    \end{corollary}

    \begin{proof}  Given $F \in (H^{\infty}_{1})^{k \times k}$,  the
        distance of $F$ to $\cI_{{\mathfrak D}}$ can be identified with
        $$
        \operatorname{dist} (F, \cI_{{\mathfrak D}}) =
        \inf \{ \| G\| \colon G\in (H^{\infty}_{1}) \text{ and }
        G(z_{j}) = F(z_{j}) \; \text{ for } \; j = 1, \dots, n\}.
        $$
        By rescaling the result in Theorem \ref{T:1}, we see that this
        infimum in turn can be computed as
        \begin{align*}
       & \inf\{ M \colon
        \sum_{i,j=1}^{n} \operatorname{trace} \left[M^{2} \cdot
        X_{j} K^{\alpha, \beta}(z_{i}, z_{j}) X_{i}^{*} -
      W_{j}^{*} X_{j} K^{\alpha, \beta}(z_{i}, z_{j}) X_{i}^{*}
      W_{i} \right] \ge 0 \\
      & \quad \text{ for all } (\alpha, \beta) \in {\mathbb
      G}(\ell' \times \ell),\, X_{1}, \dots, X_{n} \in {\mathbb
      C}^{k \times \ell},\,
      1 \le \ell \le \ell' \le k\}.
    \end{align*}
        By a routine rescaling of the  equivalence of the kernel criterion
        \eqref{op-ker-family} and the operator-norm criterion
        \eqref{op-form-ker-fam}, this infimum in turn is the same as
        $$ \inf \left\{ M \colon \left\| \left. (M_{F})^{*}
     \right|_{\cM^{{\mathfrak D}}_{\alpha, \beta}} \right\|\le M
     \text{ for all } (\alpha, \beta) \in {\mathbb G}(\ell' \times
     \ell) \text{ for } 1 \le \ell \le \ell' \le k \right\}.
     $$
     This last formula finally is clearly equal to the supremum given in
     \eqref{distance}.
      \end{proof}

  \subsection{A Beurling-Lax theorem for $(H^{\infty}_{1})^{k \times
  k}$}

  We begin with the following variant of the classical Beurling-Lax
  theorem.

  \begin{theorem}  \label{T:BLclassical}
      Suppose that $\cM \subset (H^{2})^{k \times k}$ is invariant
      under the multiplication operator $M_{F} \colon h(z) \mapsto F(z)
      h(z)$ for each $F \in (H^{\infty})^{k \times k}$.  Then, for
      some $\ell \le k$ there is an $\ell \times k$ co-inner function $J$
      (so $J \in (H^{\infty})^{\ell \times k}$ with $J(z) J(z)^{*} =
      I_{\ell}$ for almost all $z \in {\mathbb T}$) so that
      $$
       \cM = (H^{2})^{k \times \ell} \cdot J.
      $$
   \end{theorem}

   \begin{proof}
       Let $\cE$ be the subspace
       $$
       \cE = \cM \ominus (z {\mathbb C}^{k \times k}) \cM.
       $$
       Note that $\cE$ is invariant under multiplication on the left
       by elements of ${\mathbb C}^{k \times k}$. It follows that
       there is a isometric right multiplication operator  $R_{J}$
      which is an isometry from ${\mathbb C}^{k \times \ell}$ onto
      $\cE$, i.e., so that $\cE = {\mathbb C}^{k \times \ell} J$.  As $\cE
       \subset (H^{2})^{k \times k}$, the entries of $J$ are
       analytic.  One can check that $\cE$ is wandering for $z{\mathbb
       C}^{k \times k}$, i.e., $z^{n}{\mathbb C}^{k \times k} \cE
       \perp z^{m} {\mathbb C}^{k \times k} \cE$ for $n \ne m$.  As
       $\cap_{n \ge 0} z^{n}{\mathbb C}^{k \times k} \cM = \{0\}$, we
       conclude that $\cM$ has the orthogonal decomposition
       $$
       \cM = \bigoplus_{n=0}^{\infty} z^{n} {\mathbb C}^{k \times k} J
       = (H^{2})^{k \times \ell} J.
       $$
      In particular, for all $X,Y \in {\mathbb C}^{k \times k}$ and
      for all $n \in {\mathbb Z}$ not equal to $0$ we have
    \begin{align*}
    0 & = \langle z^{n} X J(z), Y J(z) \rangle = \frac{1}{2 \pi}
    \int_{{\mathbb T}} \operatorname{trace} \left[ z^{n} X J(z) J(z)^{*}
    Y^{*} \right] \, |dz|  \\
    & = \frac{1}{2 \pi} \int_{{\mathbb T}} z^{n}
    \operatorname{trace} \left[J(z) J(z)^{*} Y^{*} X \right]\, |dz|
    \end{align*}
    which forces $J(z) J(z)^{*}$ to be a constant.  As we arranged
    that right multiplication by $J$ is an isometry, the constant must
    be $I_{\ell}$ and Theorem \ref{T:BLclassical} follows.
     \end{proof}

     The Beurling-Lax theorem for the algebra $(H^{\infty}_{1})^{k \times
     k}$ is as follows.
     For this theorem we introduce the notation
     $\overline{\mathbb G}( \ell' \times \ell)$ for the set of pairs $(\alpha,
     \beta)$ of $\ell' \times \ell$ matrices with the property that
     $\alpha \alpha^{*} + \beta \beta^{*} = I_{\ell'}$, i.e.,
     $\overline{\mathbb G}( \ell' \times \ell)$ is defined as
     $\mathbb G( \ell' \times \ell)$ but without the injectivity condition
     on $\alpha$.

     \begin{theorem} \label{T:BLHinf1}  Suppose that the nonzero subspace
    $\cM$ of $(H^{2})^{k \times k}$ is invariant under $M_{F}
    \colon h(z) \mapsto F(z) h(z)$ for all $F \in
    (H^{\infty}_{1})^{k \times k}$.  Then there is a co-inner
    function $J \in (H^{\infty})^{\ell \times k}$ and an
    $(\alpha, \beta) \in \overline{\mathbb G}(\ell' \times \ell)$ for some
    $1 \le\ell \le \ell' \le k$ so that $\cM = H^{2}_{\alpha, \beta}
    \cdot J$.
   \end{theorem}

   \begin{proof}
   Let $\widetilde \cM := \left[
   (H^{\infty})^{k \times k} \cdot \cM \right]^{-}$.  Then
   $$
    \widetilde \cM \supset \cM \supset (H^{\infty}_{1})^{k \times k}
    \cdot \cM \supset \left[ z^{2} (H^{\infty})^{k \times k} \cdot \cM
    \right]^{-} = z^{2} \cM.
   $$
   By Theorem \ref{T:BLclassical}, there is an $\ell
   \times k$ co-inner function $J$ so that $\widetilde \cM = (H^{2})^{k
   \times \ell} \cdot J$.  Then
   $$
   (H^{2})^{k \times \ell} J \supset \cM \supset z^{2} (H^{2})^{k
   \times \ell} J
   $$
   from which we see that $\cM$ has the form
   $$
     \cM = \cS \cdot J + z^{2} (H^{2})^{k \times \ell} J
   $$
   for some subspace $\cS \subset (H^{2})^{k \times \ell} \ominus
   z^{2}(H^{2})^{k \times \ell}$.  As $\cS$ is invariant under the
   operators $M_{X}$ of multiplication by a constant $k \times k$
   matrix on the left, we conclude that $\cS$ must have the form
   $$
   \cS = \{ X + zY \colon \begin{bmatrix} X & Y \end{bmatrix} = U
   \begin{bmatrix} \alpha & \beta \end{bmatrix} \text{ for some }
       (\alpha, \beta) \in \overline{\mathbb G}( \ell' \times \ell)\}
   $$
   from which it follows that $\cM = H^{2}_{\alpha, \beta} \cdot J$.
   \end{proof}

  \subsection{An analogous problem: Nevanlinna-Pick interpolation on a
  finitely-connected planar domain}

  The purpose of this section is to make explicit how the commutant
  lifting theorem from \cite{Ball79} leads to a proof of Theorem
  \ref{T:Ball79}.  Toward this goal we first need a few preliminaries.

  Let $\cR$ be a bounded finitely-connected planar domain with the boundary
  consisting of $g+1$ analytic closed simple curves $C_{0}, C_{1}, \dots, C_{g}$
  where $C_{0}$ is the boundary of the unbounded component of
  ${\mathbb C} \setminus \cR$.  There is a standard procedure (see
  \cite{Abrahamse}) for introducing $g$ disjoint simple curves
  $\gamma_{1}, \dots, \gamma_{g}$ so that $R \setminus {\boldsymbol
  \gamma}$ (where we set ${\boldsymbol \gamma}$ equal to the union
  ${\boldsymbol \gamma} = \gamma_{1} \cup
  \cdots \cup \gamma_{g}$) is simply connected.  For each cut
  $\gamma_{i}$ we assign an arbitrary orientation, so that points $z$
  not on $\gamma_{i}$ but in a sufficiently small neighborhood of
  $\gamma_{i}$ in $\cR$ can be assigned a location of either ``to the
  left'' of $\gamma_{i}$ or ``to the right'' of $\gamma_{i}$.
  For $f$ a function
  on $\cR \setminus {\boldsymbol \gamma}$ and $z \in \gamma_{i}$, we
  let $f(z_{-})$ denote the limit of $f(\zeta)$ as $\zeta$ approaches
  $z$ from the left of $\gamma_{i}$ in $\cR$, and similarly we let
  $f(z_{+})$ denote the limit of $f(\zeta)$ as $\zeta$ approaches $z$
  from the right of $\gamma_{i}$ in $\cR$, whenever these limits exist.
  We also fix a point $t \in R$ and let $dm_{t}(z)$ be the harmonic
  measure on $\partial R$ for the point $t \in R$; thus
  $$
     u(t) = \int_{\partial R} u(\zeta)\, dm_{t}(\zeta)
   $$
   whenever $u$ is harmonic on $\cR$ and continuous on the closure
   $\overline{R}$.

  Suppose that we are given a $g$-tuple $\bU = (U_{1}, \dots, U_{g})$
  of $ \ell \times \ell$ unitary matrices, denoted as $\bU \in
  \cU(\ell)^{g}$. We also fix a positive integer $k$.  We let
  $H^{2}_{\cC^{2, k \times \ell}}({\mathbf I}_{k},\bU)$ denote the
  space of all $(k \times \ell)$-matrix-valued holomorphic functions $F$
  on $R\setminus {\boldsymbol \gamma}$ such that
  $$
    F(z_{-}) = F(z_{+})U_{i} \text{ for } z \in \gamma_{i}
  $$
  (so $\operatorname{trace}(F(z)^{*}F(z))$ extends by continuity to a
  single-valued function on all of $\cR$) and such that
  $$
  \| F \|_{2} : = \int_{\partial R} \operatorname{trace} \left(
  F(z)^{*} F(z) \right) \, dm_{t}(z) < \infty.
  $$
  Then $H^{2}_{\cC^{2, k \times \ell}}({\mathbf I}_{k}, \bU)$ is a
  Hilbert space.  Given a bounded $k \times k$-matrix-valued function
  $F$ on $\cR$ (denoted as $F \in (H^{\infty}(\cR))^{k \times k}$) and given any $\bU
  \in \cU(\ell)^{g}$ as above, we may consider the multiplication
  operator $M^{\bU}_{F}$ on $H^{2}_{\cC^{2, k \times \ell}}({\mathbf
  I}_{k}, \bU)$ given by
  $$
    M^{\bU}_{F} \colon f(z) \mapsto F(z) f(z) \text{ for } f \in
    H^{2}_{\cC^{2, k \times \ell}}({\mathbf I}_{k}, \bU).
  $$
  Then it can be shown that
  $$
  \| F \|_{\infty} = \sup_{k \colon 1 \le k \le \ell} \quad  \sup_{\bU \in
  \cU(k)^{g}} \left\| M_{F}^{\bU} \right\|.
  $$
  For the particular case where $\bU = (I_{k}, \dots, I_{k})$
  has all entries equal to the $k \times k$ identity matrix $I_{k}$,
  we write ${\mathbf I}_{k}$ for $\bU$.  In case $k=1$ we write
  ${\mathbf 1}$ for $(1, \dots, 1) \in \cU(1)^{g}$.

  Let us now assume that we are given the data set
  $$
    z_{1}, \dots, z_{n} \in R, \quad W_{1}, \dots, W_{n} \in {\mathbb
    C}^{k \times k}
  $$
  for an interpolation problem \eqref{R-int}.  For $\bU \in
  \cU(\ell)^{g}$ we define subspaces $\cM^{\bU}$ and $\cN^{\bU}$ of
  $H^{2}_{\cC^{2, k \times \ell}}({\mathbf I}_{k},\bU)$ by
  \begin{align*}
     & \cM^{\bU}: = \{ f \in H^{2}_{\cC^{2, k \times \ell}}
     ({\mathbf I}_{k} , \bU) \colon f(z_{i}) = 0 \text{ for } i=1, \dots, n\}, \\
    & \cN^{\bU}: = H^{2}_{\cC^{2, k \times \ell}}({\mathbf I}_{k},
    \bU) \ominus \cM^{\bU}.
  \end{align*}
  Let $F_{0}$ be any particular solution of the interpolation
  conditions \eqref{R-int}; e.g., one can take such a solution to be
  a matrix polynomial constructed by solving a scalar Lagrange
  interpolation problem for each matrix entry.   For each $\bU \in
  \cU(\ell)^{g}$ define an operator $\Gamma^{\bU}$ on $\cN^{\bU}$ by
  $$
  \Gamma^{\bU} h = P_{\cN^{\bU}}(F_{0} \cdot h) \text{ for } h \in
  \cN^{\bU}.
  $$
  It is not difficult to see that $\Gamma^{\bU}$ is independent of
  the choice of $(H^{\infty}(\cR))^{k \times k}$-solution of the
  interpolation conditions \eqref{R-int}.
  The following result
  is a more concrete expression of Theorem 4.3 and the Remark
  immediately following from \cite{Ball79}.

  \begin{theorem}  \label{T:Ball79'}  There exists an $F \in
      (H^{\infty}(\cR))^{k \times k}$ with $\|F\|_{\infty} \le 1$ so
      that
  $$
  P_{\cN^{{\mathbf I}_{k}}} M_{F}^{{\mathbf I}_{k}}|_{\cN^{{\mathbf I}_{k}}}
  = P_{\cN^{{\mathbf I}_{k}}} M_{F_{0}}^{{\mathbf I}_{k}}|_{\cN^{{\mathbf I}_{k}}},
  $$
  or, equivalently, so that $F$ satisfies the interpolation conditions
  \eqref{R-int}, if and only if
  \begin{equation}  \label{CLT-criterion}
      \sup_{\ell \colon 1 \le \ell \le k} \sup_{\bU \colon \bU \in
      \cU(\ell)^{g}} \left\| \Gamma^{\bU}\right\| \le 1.
   \end{equation}
   \end{theorem}

   To convert the criterion \eqref{CLT-criterion} to a positivity
       test, we proceed as follows.  For $\bU \in \cU(\ell)^{g}$ there is
       a $(\ell \times \ell)$-matrix-valued kernel function
       $K_{\bU}(z,w)$ defined on $\cR \times \cR$ so that
       \begin{align*}
       & x K_{\bU}(\cdot, w) \in H^{2}_{\cC^{2,1 \times
       \ell}}({\mathbf 1}, \bU) \text{ for each } x \in {\mathbb
       C}^{1 \times \ell}, \\
    & \langle f, x K_{\bU}(\cdot, w) \rangle_{H^{2}_{\cC^{2, 1 \times
    \ell}}({\mathbf 1}, \bU)} = \langle f(w), x \rangle_{\cC^{2,1
    \times \ell}} = f(w) x^{*}.
      \end{align*}
      Then this same kernel $K_{\bU}(z,w)$ serves as the kernel function
      for the space of $(k \times \ell)$-matrix-valued functions
      $H^{2}_{\cC^{2, k \times \ell}}({\mathbf I}_{k},
      \bU)$ in the sense that for all $f \in H^{2}_{\cC^{2, k \times \ell}}({\mathbf
       I}_{k}, \bU)$ and $X \in {\mathbb C}^{k \times \ell}$ we have
       $X K_{\bU}(\cdot, w) \in H^{2}_{\cC^{2, k \times \ell}}({\mathbf I}_{k}, \bU)$
       and
       \[
\langle f, X K_{\bU}( \cdot, w) \rangle_{H^{2}_{\cC^{2, k \times
       \ell}}({\mathbf I}_{k}, \bU)} = \langle f(w), X \rangle_{\cC^{2,k
       \times \ell}} = \operatorname{trace} \left( f(w) X^{*} \right).
       \]
     Then one can check that the following hold:
     \begin{align}
    & \overline{\operatorname{span}} \left\{ X K_{\bU}(\cdot, z_{j})
     \colon X \in {\mathbb C}^{k \times \ell},\, j=1, \dots, n
     \right\} =
     \cN^{\bU}, \label{span} \\
     & \left( \Gamma^{\bU}\right)^{*} \colon
      X K_{\bU}( \cdot, z_{j}) \mapsto W_{j}^{*} X K_{\bU}(\cdot, z_{j}).
      \label{Gamma-action}
      \end{align}

     With these preliminaries out of the way, Theorem \ref{T:Ball79}
     follows as a consequence of Theorem \ref{T:Ball79'} by the following
     standard computations.

 \begin{proof}[Proof of Theorem \ref{T:Ball79} via Theorem
     \ref{T:Ball79'}]  By \eqref{span}, a generic element $f$ of
     $\cN^{\bU}$ is given by
     $f(z) = \sum_{j=1}^{n} X_{j} K_{\bU}(\cdot, z_{j})$
  where $X_{1}, \dots, X_{g} \in {\mathbb C}^{k \times \ell}$.  By
  \eqref{Gamma-action},
  $$
  \left( \Gamma^{\bU}\right)^{*} \colon \sum_{j=1}^{n} X_{j} K_{\bU}(
  \cdot, z_{j}) \mapsto \sum_{j=1}^{n} W_{j}^{*} X_{j} K_{\bU}(\cdot,
  z_{j}).
  $$
  Hence $\left\| \Gamma^{\bU} \right\| = \left\| \left( \Gamma^{\bU}
  \right)^{*} \right\| \le 1$ if and only if
  \begin{equation}   \label{normbd1}
      \left\| \sum_{j=1}^{n} X_{j} K_{\bU}( \cdot, z_{j})
      \right\|^{2} - \left\| \sum_{j=1}^{n} W_{j}^{*} X_{j}
      K_{\bU}(\cdot, z_{j}) \right\|^{2} \text{ for all } X_{1},
      \dots, X_{n} \in {\mathbb C}^{k \times \ell}.
  \end{equation}
  Use of the reproducing kernel property of $K_{\bU}(z,w)$ converts
  \eqref{normbd1} to \eqref{R-matrixPick}.  This holding true for all
  $\bU \in \cU(\ell)^{g}$ for each $\ell$ with $1 \le \ell \le k$ is
  equivalent to the existence of a solution of the interpolation
  problem of norm at most 1 by Theorem \ref{T:Ball79'}
  \end{proof}

  \begin{remark} {\em
      If $\bU = (U_{1}, \dots, U_{g}) \in \cU(\ell)^{g}$ has a direct
      sum decomposition
      $ \bU = \bU' \oplus \bU''$  where
    $$  \bU' = (U_{1}', \dots, U_{g}') \in \cU(\ell')^{g} \text{ and }
      \bU'' = (U_{1}'', \dots, U_{g}'') \in \cU(\ell'')^{g}
      $$
      where  $\ell = \ell' + \ell''$, then it is easily seen that
      we have the orthogonal direct-sum decomposition
   $$
   H^{2}_{\cC^{2,k \times \ell}}¥({\mathbf I}_{k}, \bU) =
   H^{2}_{\cC^{2, k \times \ell'}}({\mathbf I}_{k}, \bU')
   \oplus H^{2}_{\cC^{2, k \times \ell''}}({\mathbf I}_{k}, \bU'')
   $$
   which splits the action of $M_{F}^{\bU}$:
   $$
     M_{F}^{\bU} = M_{F}^{\bU'} \oplus M_{F}^{\bU'''}.
   $$
   It follows that in Theorem \ref{T:Ball79} we need only consider
   irreducible elements $\bU$ of $\cU(\ell)^{g}$ for each $\ell = 1,
   \dots, k$, or, alternatively, we need consider only $\cU(\ell)^{g}$
   with $\ell = k$.

   In particular, for the case of $g=1$ where each element $\bU \in
   \cU(k)^{1}$ reduces to a single unitary operator $U$, a
   consequence of the spectral theorem is that any such $U$ is a
   direct sum of scalar $U$'s ($k=1$).  Hence, for the case $g = 1$
   in Theorem \ref{T:Ball79}, it suffices to  consider the condition
   \eqref{R-matrixPick} with $U = \lambda I_{k}$ where $\lambda$ is
   on the unit circle ${\mathbb T}$.
   McCullough in \cite{McC} showed that in fact all these
  kernel functions (or more precisely, a dense set) are needed, even
  when the interpolation nodes $z_{1}, \dots, z_{n}$ are specified
  (see \cite{ McCP} for a refinement of this result).  Remarkably, for
  the case of scalar interpolation ($k = 1$), if one specifies the
  interpolation nodes, Fedorov-Vinnikov \cite{FV} (see \cite{McC} for
  an alternate proof) showed that one can find two points on ${\mathbb T}$
  for which positive-semidefiniteness of the associated Pick matrix
  guarantees solvability for the Nevanlinna-Pick interpolation problem.
  }\end{remark}

\section{Proof of Theorem \ref{T:2}}  \label{S:T2}

In this section we prove Theorem \ref{T:2} in the following more
general setting. Let $\lambda_1,\ldots,\lambda_m$ be distinct points
in $\bbD$ and $r_1,\ldots,r_m\in\bbZ_{+}$ so that $r_i \ge 1$ for
$i=1,\ldots,m$. We write $B$ for the finite Blaschke product
\begin{equation}\label{Blaschke}
B(z)=\prod_{i=1}^m\left(\frac{z-\lambda_i}{1-\overline{\lambda}_iz}\right)^{r_i}.
\end{equation}
Let $(H^\infty_B)^{k\times k}$ be the subalgebra of $(H^\infty)^{k\times
k}$
consisting of the matrix-functions in $(H^\infty)^{k\times k}$ whose
entries
are in the subalgebra $\C+BH^\infty$ of $H^\infty$, and let
$(\cS_B)^{k\times k}$
be the constrained Schur class of functions in $(H^\infty_B)^{k\times
k}$ of supremum
norm at most one.

Then we consider the following matrix-valued constrained
Nevanlinna-Pick interpolation
problem: {\em given a data set ${\mathfrak D}$ as in (\ref{op-data}),
find a function
$S$ in $(\cS_B)^{k\times k}$ satisfying the interpolation conditions
\begin{equation}\label{intercons}
S(z_i)=W_{i}\quad\text{for}\quad i=1,\ldots,n.
\end{equation}
}
We will assume that $z_i\not=\lambda_j$ for $i=1,\ldots,n$ and
$j=1,\ldots,m$,
and return to the case where $z_i=\lambda_j$ for some $i$ and $j$ at
the end of the
section.

To state the solution criterion for the constrained Nevanlinna-Pick
interpolation problem we introduce some more notations.
Define $Z,\widetilde W\in\C^{kn\times kn}$ and $E\in\C^{kn\times k}$ by
\begin{equation}\label{mats1}
Z=\mat{ccc}{z_1I_k&&\\&\ddots&\\&&z_nI_k},\quad
E=\mat{c}{I_k\\\vdots\\I_k},\quad
\widetilde W=\mat{ccc}{W_1&&\\&\ddots&\\&&W_n},
\end{equation}
and define $J\in\C^{kd\times kd}$ and $\widetilde E\in\C^{kd\times
k}$, where $d=r_1+\cdots+r_m$, by
\begin{equation}\label{mats2}
J=\mat{ccc}{J_1&&\\&\ddots&\\&&J_m},\quad \widetilde
E=\mat{c}{\widetilde E_1\\\vdots\\
\widetilde E_m},
\end{equation}
where for $i=1,\ldots,m$ the matrices $J_i\in\C^{kr_i\times kr_i}$
and $\widetilde E_i\in\C^{kr_i\times k}$
are given by
\begin{equation}\label{mats3}
J_i=\mat{ccccc}{
\lambda_iI_k&0&\cdots&\cdots&0\\
I_k&\lambda_iI_k&\ddots&&\vdots\\
0&\ddots&\ddots&\ddots&\vdots\\
\vdots&\ddots&\ddots&\ddots&0\\
0&\cdots&0&I_k&\lambda_iI_k},\quad
\widetilde E_i=\mat{c}{I_k\\0\\\vdots\\\vdots\\0}.
\end{equation}
Let $\Pick$ be the Pick matrix associated with the standard matrix-valued
Nevanlinna-Pick problem, i.e., as in (\ref{pick-mat}), and let
$\Q\in\C^{kd\times kd}$ and
$\widetilde\Q\in\C^{kd\times kn}$ be the unique solutions to the
Stein equations
\begin{equation}\label{Stein}
\Q-J\Q J^*=\widetilde E\widetilde E^*,\qquad \widetilde\Q- J\widetilde\Q
Z^*=\widetilde EE^*,
\end{equation}
that is, $\Q$ and $\widetilde\Q$ are given by the convergent infinite
series
\begin{equation}\label{QtildeQ}
\Q=\sum_{i=0}^\infty J^i\widetilde E\widetilde E^*
J^{*i}\quad\text{and}\quad
\widetilde\Q=\sum_{i=0}^\infty J^i\widetilde E E^* Z^{*i}.
\end{equation}
For the special case considered here there are more explicit formulas
available
  (see e.g.~Appendix A.2 in \cite{BGR}) but we have no need for them.
  In particular one can see directly from the series expansion for $\Q$
  in \eqref{QtildeQ} or quote general theory (using the fact that the
  pair $(J,\widetilde E)$ is controllable---see \cite{BGR})
  that $\Q$ is positive definite and hence invertible.  In fact from
  the infinite series in \eqref{QtildeQ} one can see that $Q \geq I$.

For $X\in\C^{k\times k}$ and $j\in\bbZ_+$ set $X_j\in\C^{jk\times jk}$,
\[
X_j=\mat{ccc}{X&&\\&\ddots&\\&&X},
\]
and define $\Pick_X\in\C^{k(n+2d)\times k(n+2d)}$ by
\begin{equation}\label{pick-X2}
\Pick_X=\mat{ccc}{
\Pick&(I-\widetilde WX_n^*)\widetilde\Q^*&0\\
\widetilde\Q(I-X_n\widetilde W^*)&\Q&X_d\\
0&X_d^*&\Q^{-1}}.
\end{equation}
Then we have the following result.

\begin{theorem}\label{T:3}
Suppose we are given an interpolation data set of the form ${\mathfrak D}$
as in (\ref{op-data}) with $W_1,\ldots,W_n$ equal to $k\times k$ matrices
along with a Blaschke product $B$ as in (\ref{Blaschke}). Assume that
$z_i\not=\lambda_j$ for $i=1,\ldots,n$ and $j=1,\ldots,m$. Then there
exists
a constrained Schur-class function $S$ in $(\cS_B)^{k\times k}$ satisfying
the interpolation conditions (\ref{intercons}) if and only if there exists
a $k\times k$ matrix $X$ so that the matrix $\Pick_X$ in (\ref{pick-X2})
is
positive-semidefinite.
\end{theorem}

\begin{proof}
Notice that an (unconstrained) Schur class function $S$ in
$(H^\infty)^{k\times k}$, i.e.,
$\|S\|_\infty\,\leq1$, is in the constrained Schur class
$(\cS_B)^{k\times k}$ if and only if
\begin{eqnarray}
\label{addcons1}
&\displaystyle S^{(j)}(\lambda_i)=0\ \
\text{for}\ i=1,\ldots,m,\ \text{and}\ j=1,\ldots,r_j-1,&\\
\label{addcons2}
&\displaystyle S(\lambda_1)=\cdots=S(\lambda_m).&
\end{eqnarray}
(If $r_{j}=1$ for some $j$, we interpret \eqref{addcons1} as
vacuous.)
So, alternatively we seek Schur class functions $S$ satisfying the
constraints (\ref{intercons}), (\ref{addcons1}) and (\ref{addcons2}).

Now assume that $S$ is a solution to the constrained Nevanlinna-Pick
interpolation problem.
Set $X=S(\lambda_1)$. Then $X$ is contractive, and $S$ is also
a solution to the matrix-valued Carath\'eodory-Fej\'er interpolation
problem:
{\em find Schur class functions $S$ in $(H^\infty)^{k\times k}$
satisfying (\ref{intercons}),
(\ref{addcons1}) and
\begin{equation}\label{CFaddcons}
S(\lambda_i)=X\quad\text{for}\quad i=1,\ldots,m.
\end{equation}}
Conversely, for any contractive $k\times k$ matrix $X$, a Schur class
function in $(H^\infty)^{k\times k}$ satisfying (\ref{intercons}),
(\ref{addcons1}) and (\ref{CFaddcons}) is obviously also a solution to
the constrained Nevanlinna-Pick problem.

{}From this we conclude that the solutions of the constrained
Nevanlinna-Pick problem
correspond to the solutions of the Carath\'eodory-Fej\'er problem
defined by the
interpolation conditions (\ref{intercons}), (\ref{addcons1}) and
(\ref{CFaddcons}),
where $X$ is a free-parameter in the set of (contractive) $k\times k$
matrices.

Given a $k\times k$ matrix $X$, the Carath\'eodory-Fej\'er problem
described above
is known to have a solution if and only if the Pick matrix
$\widetilde\Pick_X$ given by
\begin{equation}\label{pick-tildeX2}
\widetilde\Pick_X=\mat{cc}{\Pick&(I-\widetilde
WX_n^*)\widetilde\Q^*\\ \widetilde\Q(I-X_n\widetilde W^*)&\Q-X_d \Q X_d^*}
\end{equation}
is positive-semidefinite; see \cite{BGR}. Here we use the same
notations as in the definition of the matrix
$\Pick_X$ in (\ref{pick-X2}).
The theorem then follows from the fact that $\Q$ is invertible and
$\widetilde\Pick_X$ can be written as
\[
\widetilde\Pick_X=\mat{cc}{\Pick&(I-\widetilde
WX_n^*)\widetilde\Q^*\\\widetilde\Q(I-X_n\widetilde W^*)&\Q}
-\mat{c}{0\\X_d}\Q\mat{cc}{0&X_d^*}
\]
so that, with a standard Schur complement argument (see \cite[Remark
I.1.2]{BGK}), we see that for each
$X\in\C^{k\times k}$ the matrix $\widetilde\Pick_X$ is
positive-semidefinite if and only if
the matrix $\Pick_X$ in (\ref{pick-X2}) is positive-semidefinite.
\end{proof}

\begin{proof}[Proof of Theorem \ref{T:2}]
The problem considered in Theorem \ref{T:2} is the special case of the
problem
considered in this section where $m=1$, $r_1=2$ and $\lambda_1=0$,
i.e., $B(z)=z^{2}$.
In that case the solutions $\Q$ and $\widetilde\Q$ to the Stein
equations (\ref{Stein}) are
given by
\[
\Q=\mat{cc}{I_k&0\\0&I_k}\quad\text{and}\quad
\widetilde\Q=\mat{ccc}{I_k&\cdots&I_k\\\overline{z}_1I_k&\cdots&\overline{z}_nI_k},
\]
so that
\[
\widetilde\Q(I-X_n\widetilde W^*)=\mat{c}{E^*-XW^*\\(E^*-XW^*)Z^*}
\]
with $W$ as in (\ref{ZEW}).
It thus follows that in this special case the matrix $\Pick_X$ in
(\ref{pick-X2}) reduces
to (\ref{Pick-X}).
\end{proof}

We conclude this section with some remarks concerning Theorem \ref{T:3}.

\subsection{Parametrization of the set of all solutions}

For each $X\in\C^{k\times k}$ so that $\Pick_X$ in (\ref{pick-X2}) is
positive-semidefinite, the
solutions to the Carath\'eodory-Fej\'er problem defined by the
interpolation
conditions (\ref{intercons}), (\ref{addcons1}) and (\ref{CFaddcons}) can
be
described explicitly as the image of a
linear-fractional-transformation $T_{\Theta_{X}}$ acting on the unit ball
of matrix-valued $H^{\infty}$ of some size. Therefore, if one happens
to be able to find all
$X\in\C^{k\times k}$ for which $\Pick_X$ is positive-semidefinite,
then in principle  one can
describe the set of all solutions to the constrained Nevanlinna-Pick
problem as the union of the images of these
linear-fractional-transformations $T_{\Theta_{X}}$.

\subsection{The case where $B$ and $B_{{\mathfrak D}}$ have
overlapping zeros}

Theorem \ref{T:3} (and also Theorem \ref{T:2}) only considers the case that
$\{\lambda_1,\ldots,\lambda_m\}$ and $\{z_1,\ldots,z_n\}$ do not
intersect.
In case the intersection is not empty, say
\[
\{\lambda_1,\ldots,\lambda_m\} \cap
\{z_1,\ldots,z_n\}=\{z_{i_1},\ldots,z_{i_p}\},
\]
then solutions exist if and only if $W_{i_1}=\cdots=W_{i_p}$ and
$\Pick_{W_{i_1}}$
is positive-semidefinite.
Indeed, this is the case since functions in the constrained
Schur class $(\cS_B)^{k\times k}$ have to satisfy (\ref{addcons2}) and for
a
solution $S$ the matrix $\Pick_{S(\lambda_1)}$ is positive-semidefinite.

The Pick matrix obtained in this case includes some degeneracy since
the interpolation
condition at $z_{i_j}$ is listed twice. There therefore exists a
reduced Pick matrix
whose size, for the case $k=1$, is still larger than the number of
interpolation points
$n$. In \cite{Solazzo} the size of this reduced Pick matrix is
identified with the
dimension of the $C^*$-envelope $C^*({\mathfrak A})$ for the algebra
${\mathfrak A}=H^\infty_B/\cI_{\mathfrak D}$, where $\cI_{\mathfrak
D}$ is the ideal
\[
\cI_{{\mathfrak D}} :=
\{f\in (H^{\infty}_{B}) \colon f(z_{i}) = 0 \text{ for } i=1, \dots, n\}.
\]

\subsection{The criterion in Theorem \ref{T:2DPRS}}

If $X\in\C^{k\times k}$ is a strict contraction, then $\Q-X_d\Q
X_d^*$ is invertible
(if $X$ is not a strict contraction one can use a Moore-Penrose inverse),
and,
again with a Schur complement argument, it follows that
$\widetilde\Pick_X$ in
(\ref{pick-tildeX2}) is positive-semidefinite if and only if the matrix
\begin{equation}\label{}
{\widehat \Pick}_X:=
\Pick-(I-\widetilde W X_n^*)\widetilde\Q^*(\Q-X_d\Q
X_d^*)^{-1}\widetilde\Q(I-X_n\widetilde W^*)
\end{equation}
is positive-semidefinite.
When specified for the setting considered in \cite{DPRS}, i.e., as in
Theorems
\ref{T:1DPRS} and \ref{T:2DPRS}, the matrix ${\widehat \Pick}_\lambda$
is conjugate to the matrix (\ref{clas-crit1}). More specifically, the
matrix
(\ref{clas-crit1}) is equal to $T{\widehat \Pick}_\lambda T^*$ with
\[
T=\mat{ccc}{
\frac{1-\overline{\lambda}w_1}{\sqrt{1-|\lambda|^2}}&&\\
&\ddots&\\&&\frac{1-\overline{\lambda}w_n}{\sqrt{1-|\lambda|^2}}}.
\]
Thus the matrix (\ref{clas-crit1}) being positive-semidefinite corresponds
to
$\Pick_\lambda$ in (\ref{Pick-X}) (for $k=1$) positive-semidefinite.

\subsection{More general algebras}

The argument used in the proof of Theorem \ref{T:3} extends to more
general
subalgebras of $H^\infty$.

Let $B_1,\ldots,B_l$ be finite Blaschke products with $B_i$ having zeros
$\lambda_1^{(i)},\ldots,\lambda_{m_i}^{(i)}$ for $i=1,\ldots l$. Then one
can
consider Nevanlinna-Pick interpolation for functions in the intersection
of the
algebras $H^\infty_{B_i}$. In case all zeros $\lambda_j^{(i)}$ are
distinct
this gives rise to a criterion where one has to check
positive-semidefiniteness
of a Pick matrix with $l$ free parameters.
Given a finite Blaschke product $B$ and a finite number of
polynomials $p_1,\ldots,p_l$,
let ${\mathfrak A}_{p_1,\ldots,p_l}$ be the subalgebra of $H^\infty$
generated by
$p_1,\ldots,p_l$. Nevanlinna-Pick interpolation for functions in the
subalgebra
${\mathfrak A}_{p_1,\ldots,p_l}+B H^\infty$ gives rise to coupling
conditions more
complicated than (\ref{addcons2}) and more intricate Pick matrices
than (\ref{pick-X2})
with possibly more then one free parameter.

For example, consider the case of  three distinct points $\lambda_1$,
$\lambda_2$ and
$\lambda_3$ in $\bbD$ with associated the Blaschke product
\[
B(z)=\left(\frac{z-\lambda_1}{1-\overline{\lambda}_1z}\right)^{3}
\left(\frac{z-\lambda_2}{1-\overline{\lambda}_2z}\right)^{3}
\left(\frac{z-\lambda_3}{1-\overline{\lambda}_3z}\right)^{2}
\]
and polynomials $p_1(z)=1$ and $p_2(z)=z^{2}$.
Those functions $f$ which are in the algebra ${\mathfrak
A}_{p_1,p_2}+B H^\infty$
can then be characterized as those $f$ in $H^\infty$ that satisfy
$f'(\lambda_i)=0$ for $i=1,2,3$
along with
\begin{equation}\label{coupling}
\begin{array}{c}
f''(\lambda_1)=f''(\lambda_2),\\[.2cm]
2f(\lambda_1)+\lambda_1^2f''(\lambda_1)=2f(\lambda_2)+\lambda_2^2f''(\lambda_2)
=2f(\lambda_3)+\lambda_3^2f''(\lambda_3).
\end{array}
\end{equation}
By specifying both values in the coupling conditions (\ref{coupling})
we return to a standard Carath\'eodory-Fej\'er problem;  thus,
for the general problem we obtain a Pick matrix with two free parameters.

\subsection{The non-square case}
We expect that the techniques used
here can enable one to obtain a non-square version of Theorems
\ref{T:1} and \ref{T:3} where one seeks $S \in (\cS_{1})^{k \times
k'}$ (or, more generally, $S \in (\cS_{B})^{k \times k'}$) with $k
\ne k'$ which satisfies some prescribed set of interpolation
conditions. Note that $(\cS_{B})^{k \times k'}$ is no longer an
algebra but rather an {\em operator space} when $k \ne k'$ (see
e.g.~\cite{Pisier}). There is an abstract notion of a non-square
analogue of a $C^{*}$-algebra, namely the $J^{*}$-algebras
introduced by Harris \cite{Harris}. Perhaps looking for the
$J^{*}$-algebra envelop of a quotient operator space will give
insight into interpolation and dilation theory, as is the case for
the algebra setting (see \cite{DPRS}).

\section{Analysis of the LMI criterion of Theorem \ref{T:2}}\label{S:LMI}

Recall from the introduction that the LMI criterion of Theorem \ref{T:2} is equivalent to
the existence of a matrix $X$ such that the Pick matrix $\Pick'_X$ given by (\ref{Pick-X'})
be positive semidefinite, i.e.,
\begin{equation}\label{criterion1}
\mat{cc}{\Pick&\tilE+\tilW\tilX^*\\\tilE^*+\tilX\tilW^*&I-\tilX\tilX^*}\geq0,
\end{equation}
where $\Pick$ is the standard Pick matrix (\ref{pick-mat}),
\begin{equation}\label{specified}
\tilE=\mat{cc}{E&ZE},\quad\tilW=\mat{cc}{W&ZW}\quad\text{and}\quad\tilX=\mat{cc}{X&0\\0&X}
\end{equation}
with $E$, $Z$ and $W$ as in (\ref{ZEW}).

We forget for now about the structure in the matrices $\tilE$, $\tilW$ and $\tilX$, and just
assume that $\tilE$, $\tilW$ are given $k\times nk$ matrices and $\tilX$ is a free-parameter
$2k\times 2k$ matrix.
In case $\Pick$ is positive definite we obtain, after taking the Schur complement with respect
to $\Pick$, that (\ref{criterion1}) is equivalent to the
positive-semidefiniteness condition
\begin{equation}   \label{graphXpos}
\begin{array}{l}
I-\tilX\tilX^*-(\tilE^*+\tilX\tilW^*)\Pick^{-1}(\tilE+\tilW\tilX^*)=\\[.2cm]
\hspace*{2cm}=\mat{cc}{I&\tilX}
\mat{cc}{I-\tilE^*\Pick^{-1}\tilE&-\tilE^*\Pick^{-1}\tilW\\
-\tilW^*\Pick^{-1}\tilE&-(I+\tilW^*\Pick^{-1}\tilW)}\mat{c}{I\\\tilX^*} \succeq 0.
\end{array}
\end{equation}

Now assume that $\Pick$ is positive definite and that the matrix
\begin{equation}\label{M}
M=\mat{cc}{M_{11}&M_{12}\\M_{21}&M_{22}}:=
\mat{cc}{I-\tilE^*\Pick^{-1}\tilE&-\tilE^*\Pick^{-1}\tilW\\
-\tilW^*\Pick^{-1}\tilE&-(I+\tilW^*\Pick^{-1}\tilW)}
\end{equation}
is invertible.
Then \eqref{graphXpos} can be interpreted as saying that the subspace
\begin{equation}   \label{graph}
\cG_{X^{*}} = \begin{bmatrix} I \\ \tilX^{*} \end{bmatrix} {\mathbb
C}^{2k}
\end{equation}
is a $2k$-dimensional positive subspace in the Kre\u{\i}n space
$({\mathbb C}^{4k}, M)$, i.e., ${\mathbb C}^{4k}$ endowed with the
Kre\u{\i}n-space inner product
$$
[x,y]_{M} = \langle M x, y\rangle_{{\mathbb C}^{4k}}
$$
(we refer the reader to \cite{Bognar} for elementary facts concerning
Kre\u{\i}n spaces).
Conversely, any $2k$-dimensional positive subspace $\cG$ of
$({\mathbb C}^{4k},M)$ has the form \eqref{graph}  as long as
\begin{equation}  \label{compl}
    \cG \cap \begin{bmatrix} 0 \\ {\mathbb C}^{2k} \end{bmatrix} = \{0\}.
\end{equation}
Since $M_{22} = -(I + \tilW^{*}{\mathbb P}^{-1} \tilW)$ is strictly
negative definite, the subspace $\sbm{0 \\ {\mathbb C}^{2k}}$ is
uniformly negative in $({\mathbb C}^{4k},M)$ and hence condition
\eqref{compl} is automatic for any positive subspace $\cG \subset
({\mathbb C}^{4k},M)$.  We conclude that {\em there exist $(2k \times
2k)$-matrix solutions $\tilX$ to \eqref{graphXpos} if and only if the
invertible matrix $M$ has at least $2k$ positive eigenvalues.}  As
already observed above, $M_{22}< 0$ so $M$ must have at least $2k$
negative eigenvalues.  We conclude that {\em there exist solutions to
\eqref{graphXpos} if and only if $M$ has exactly $2k$ positive
eigenvalues}, or, equivalently, again by a standard Schur-complement argument
\cite[Remark I.1.2]{BGK}, {\em if and only if the Schur complement
\begin{align}
\Lambda & =I-\tilE^*\Pick^{-1}\tilE+\tilE^*\Pick^{-1}\tilW(I+\tilW^*\Pick^{-1}\tilW)^{-1}
\tilW^*\Pick^{-1}\tilE \notag \\
& = I - \tilE^{*} (\Pick + \tilW \tilW^*)^{-1} \tilE \label{Lambda}
\end{align}
is positive definite.}
Given that this is the case, we can then factor
$M$ as $M = A^{-1*} J A^{-1}$ where we set $J = \sbm{I_{2k} & 0 \\ 0
& -I_{2k}}$ and where $A = \sbm{a_{11} & a_{12} \\ a_{21} & a_{22}}$
is an invertible $4k \times 4k$ matrix with block entries $a_{ij}$
taken to have size $2k \times 2k$ for $i,j=1,2$.  Then {\em the set
$X^{*}$ for which \eqref{graphXpos} holds can be expressed as the set
of all $2k \times 2k$ matrices $X$ so that $X^{*}$ has the form
$$
  X^{*} = (a_{21}K + a_{22})(a_{11}K + a_{12})^{-1}
$$
for some contractive $2k \times 2k$ matrix $K$.}  With a little more
algebra, such an image of a linear-fractional map can be converted
to the form of a matrix ball.   Results of this type go back at
least to \cite{Siegel,Hua,Potapov55,Potapov79}. The following
summary for our situation here can be seen as a direct result of the
discussion here and Theorem 1.6.3 of \cite{DFK92}.

\begin{theorem}\label{T:DFK}
Assume that $\Pick$ is positive definite and that $M$ in \eqref{M} is invertible.
Then there exists an $2k\times 2k$ matrix $\tilX$ such that \eqref{criterion1} holds
if and only if $\Lambda$ in \eqref{Lambda} is positive semidefinite. In that case
$\Lambda$ is positive definite and all $\tilX$ such that \eqref{criterion1} holds are
given by
\begin{equation}\label{tilX}
\tilX=C+L^{\half}KR^{\half},
\end{equation}
where $K$ is a free-parameter contractive $2k\times 2k$ matrix. Here $R=\Lambda$, and
$C$ and $L$ are the $2k\times 2k$ matrices given by
\[
C= - \tilE^{*} (\Pick + \tilW \tilW^{*})^{-1} \tilW, \quad L = I -
\tilW^{*} (\Pick + \tilW \tilW^{*})^{-1} \tilW.
\]
Moreover, \eqref{criterion1} holds with strict inequality if and only if $K$ is a strict
contraction.
\end{theorem}

Under the assumptions of Theorem \ref{T:DFK} we thus get a full description of all matrices
$\tilX$ so that (\ref{criterion1}) is satisfied. This however is not enough to solve the problem
{\bf MCNP}, since the matrix $\tilX$ is also required to be of the form (\ref{specified}). We thus get
the following result.

\begin{theorem}
Let ${\mathfrak D}$ be a data set as in \eqref{op-data}. Assume that the Pick matrix $\Pick$ in
\eqref{pick-mat} is positive definite, and that $M$ in \eqref{M} is invertible.
Here $\tilE$ and $\tilW$ are given by \eqref{specified} with $E$, $W$ and $Z$ as in
\eqref{ZEW}. Then there exists a solution to the problem {\bf MCNP} if and only if
(1) $\Lambda$ in \eqref{Lambda} is positive definite and (2) there exists a contractive $2k\times 2k$
matrix $K$ such that $\tilX$ given by \eqref{tilX} is of the form in \eqref{specified}.
\end{theorem}

We now consider the LMI criterion for the problem {\bf CNP},
i.e., the Nevanlin\-na-Pick
problem for functions in the scalar-valued constrained Schur class $\cS_1$.
Let ${\mathfrak D}$ be a data set as in (\ref{clas-data}). If $|w_i|=1$ for some $i$, then a
solution exists if and only if $w_i=w_j$ for all $i,j=1,\ldots,n$. In that case, the solution
is unique and equal to the constant function with value $w_i$, and the Pick matrix $\Pick_x$ in
(\ref{Pick-X}) (with $k=1$) is positive semidefinite only for $x=w_i$. On the other hand,
if $|w_i|<1$ for all $i$, then for $\Pick_x$ to be positive semidefinite it is necessary
that $|x|<1$. In the latter case we obtain the following result.

\begin{theorem}\label{T:LMIclas}
Let ${\mathfrak D}$ be a data set as in \eqref{clas-data} with $|w_i|<1$ for all $i$.
Define $Z$, $E$ and $W$ by \eqref{ZEW} with $k=1$, and assume that $\Delta=\Pick+WW^*+ZWW^*Z$
is positive definite. Set $\widetilde\Delta$ equal to
\[
\widetilde\Delta=\Pick -EE^*-ZEE^*Z^*+(WE^*+ZWE^*Z^*)\Delta^{-1}(EW^*+ZEW^*Z^*).
\]
Then there exists a solution to the problem ${\bf CNP}$ if and only
if (1) the matrix
$\widetilde\Delta$ is positive semidefinite and (2) there exists a contractive
$n\times n$ matrix $K$ such that
\begin{equation}\label{Kcondition}
\tilX:=\Delta^{-1}(EW^*+ZEW^*Z^*)+\Delta^{-\half}K\widetilde\Delta^{\half}
\end{equation}
is a scalar multiple of the identity matrix $I_n$.
Moreover, given that  $\widetilde\Delta$ is positive semidefinite and
given $x \in {\mathbb D}$, the Pick matrix
$\Pick_x$ in \eqref{Pick-X} is positive semidefinite  if and only if $\tilX: = \overline{x} I_{n}$ has the
form \eqref{Kcondition} for some contractive $n \times n$ matrix $K$.
Finally, for the existence of an $x \in {\mathbb D}$ with $\Pick_x$ positive definite it is necessary
that $\widetilde\Delta$ be positive definite, and if this is the case, then $\Pick_x$ is positive definite
if and only if $\tilX: = \overline{x} I_{n}$ has the form \eqref{Kcondition} for some strictly contractive
$n \times n$ matrix $K$.
\end{theorem}

\begin{proof}
It follows from Theorem \ref{T:2}, specified to the case $k=1$,
that a solution to problem {\bf CNP} exists if and only if there
exists an $x\in\C$ such that the Pick matrix $\Pick_x$ in
(\ref{Pick-X}) is positive semidefinite, or equivalently,
the  Pick matrix $\Pick_x'$ in (\ref{Pick-X'}) is positive semidefinite.
In this case, the term $I-XX^*$ is just $1-|x|^2$. As explained above,
without loss of generality we may assume that $|x|<1$.

Now fix an $x\in\bbD$. We can then take the Schur complement of $\Pick_x'$
with respect to the block $\sbm{1-|x|^2&0\\0&1-|x|^2}$ to obtain that $\Pick_x'$
is positive semidefinite if and only if
\begin{equation}\label{eq1}
\Pick-\frac{1}{1-|x|^2}(E-\overline{x}W)(E^*-xW^*)-\frac{1}{1-|x|^2}Z(E-\overline{x}W)(E^*-xW^*)Z^*\geq0.
\end{equation}
After multiplication with $1-|x|^2$ and rearranging terms it follows that (\ref{eq1})
is equivalent to
\[
|x|^2\Delta-\overline{x}(WE^*+ZWE^*Z^*)-x(EW^*+ZEW^*Z^*)+\Pick-EE^*-ZEE^*Z^*\leq0
\]
and thus equivalent to
\begin{equation}\label{eq2}
\begin{array}{l}
(x\Delta^{\half}-(WE^*+ZWE^*Z^*)\Delta^{-\half})(\overline{x}\Delta^{\half}-\Delta^{-\half}(EW^*+ZEW^*Z^*))
\leq\\
\hspace*{.4cm}\leq \Pick-EE^*-ZEE^*Z^*+(WE^*+ZWE^*Z^*)\Delta^{-1}(EW^*+ZEW^*Z^*)=\tilde\Delta.
\end{array}
\end{equation}
It follows in particular that $\widetilde\Delta$ must be positive semidefinite for a
solution to exist.

Assume that $x\in\bbD$ is such that $\Pick_x'$ is positive
semidefinite, and thus also
$\widetilde\Delta$ is positive semidefinite. Then by the previous computations we see that
\begin{equation}\label{tilK}
\tilK=\overline{x}\Delta^{\half}-\Delta^{-\half}(EW^*+ZEW^*Z^*)
\end{equation}
satisfies $\tilK^*\tilK\leq\widetilde\Delta$. Using Douglas factorization lemma
we obtain a contractive matrix $K$ such that $\tilK=K\widetilde\Delta^{\half}$,
and thus
\begin{eqnarray*}
\Delta^{-1}(EW^*+ZEW^*Z^*)+\Delta^{-\half}K\widetilde\Delta^{\half}
&=&\Delta^{-\half}(\Delta^{-\half}(EW^*+ZEW^*Z^*)+\tilK)\\
&=&\Delta^{-\half}(\overline{x}\Delta^{\half})=\overline{x}I_n
\end{eqnarray*}
is a scalar multiple of the identity $I_n$.

Conversely, assume that $\tilde\Delta$ is positive semidefinite and that
there exists a contractive matrix $K$ such that (\ref{Kcondition}) is a
scalar multiple of $I_k$; say (\ref{Kcondition}) is equal to $\alpha I_k$.
Then take $x$ equal to $\overline{\alpha}$. It follows that
$\tilK=K\widetilde\Delta^{\half}$ is given by (\ref{tilK}) and satisfies
$\tilK^*\tilK\leq\widetilde\Delta$. In other words, for this choice of $x$
the inequality (\ref{eq2}) holds.
Hence the Pick matrix $\Pick_x'$ is positive semidefinite, and a solution
exists.

To verify the last statement, note that $\Pick_x'$ is positive definite if and only if the
Schur complement of $\Pick_x'$ with respect to the block $\sbm{1-|x|^2&0\\0&1-|x|^2}$ is
positive definite, which is the same as having strict inequality in (\ref{eq2}).
Since the left hand side in (\ref{eq2}) is positive semidefinite is follows right away that
$\widetilde\Delta$ must be positive definite for $\Pick_x'$ to be positive definite.
So assume $\widetilde\Delta>0$.
If $x\in{\mathbb D}$ with $\Pick_x'$ positive definite, then $\tilK$ in (\ref{tilK}) satisfies
$\tilK^*\tilK<\widetilde\Delta$, and thus $\tilK=K\widetilde\Delta^{\half}$ for some strict
contraction $K$. With the same argument as above it then follows that $\tilX=xI_n$ is given
by (\ref{Kcondition}). Conversely, if $K$ is a strictly contractive matrix such that $\tilX$
in (\ref{Kcondition}) is a scalar multiple of the identity, say $\tilX=\alpha I_n$, then
as above it follows that (\ref{eq2}) holds with $x=\overline{\alpha}$ but now with strict
inequality.
\end{proof}

As in Theorem \ref{T:DFK}, the criterion of Theorem \ref{T:LMIclas} is one of verifying
whether there exists an element in a certain matrix ball that has a specified
structure. In the special case of a constrained Nevanlinna-Pick problem with just
one point ($n=1$) the result of Theorem \ref{T:LMIclas} provides a definitive answer
to the problem {\bf CNP}.

\begin{corollary}\label{C:disk}
Let ${\mathfrak D}$ be a data set as in (\ref{clas-data}) with $n=1$ and
$|w_1|<1$. Then a solution always exists, and set of values $x$ for which
the Pick matrix $\Pick_x$ is positive semidefinite is the closed disk
$\overline{{\mathbb D}}(c; r) = \{ x \colon | x - c| \le r\}$ with
center $c$ and radius $r$ given by
\begin{equation}\label{CandR}
c=\frac{w_1(1-|z_1|^4)}{1-|z_1|^4|w_1|^2}\quad\text{and}\quad
r=\frac{|z_1|^2(1-|w_1|^2)}{1-|z_1|^4|w_1|^2}.
\end{equation}
Moreover, $\Pick_x$ is positive definite if and only if $x$ is in the open disk
${\mathbb D}(c; r) = \{ x \colon | x - c| < r\}$.
\end{corollary}

\begin{proof}
The condition that (\ref{Kcondition}) be a scalar multiple of $I_n$ is
trivially satisfied because $n=1$. It follows, after some computations, that
\[
\Delta=\frac{1-|w_1|^2|z_1|^4}{1-|z_1|^2}>0\quad\text{and}\quad
\widetilde\Delta=\left(\frac{|z_1|^2(1-|w_1|^2)}{1-|w_1|^2|z_1|^4}\right)^2\Delta>0.
\]
Thus, by Theorem \ref{T:LMIclas}, a solution exists. Moreover, the last part
of Theorem \ref{T:LMIclas} tells us that the set of $x$ so that $\Pick_x$ is positive
semidefinite is given by the complex conjugates of $\tilX$ in (\ref{Kcondition}),
where $K$ is now an element of the closed disk $\overline{\bbD}$.
Thus $\Pick_x$ is positive semidefinite for all
$x$ in the closed disk with center $c$ and radius $r$ given by
\begin{eqnarray*}
c&=&\overline{\Delta^{-1}(\overline{w_1}(1+|z_1|^2))}=\frac{w_1(1-|z_1|^4)}{1-|z_1|^4|w_1|^2},\\
r&=&\Delta^{-\half}\widetilde\Delta^{\half}=\left(\widetilde\Delta\Delta^{-1}\right)^{\half}
=\frac{|z_1|^2(1-|w_1|^2)}{1-|z_1|^4|w_1|^2}.
\end{eqnarray*}
\end{proof}

\begin{remark}\textup{
We observed that for the problem {\bf CNP} the points $x$ so that the
Pick matrix
$\Pick_x$ in (\ref{Pick-X}) is positive semidefinite must be in the
open unit disk
$\bbD$ whenever all values $w_1,\ldots,w_n$ from the data set are in $\bbD$. For the special
case that $n=1$, Corollary \ref{C:disk} tells us that the set of $x$ for which
$\Pick_x$ is positive semidefinite is given by a closed disk with center $c$ and
radius $r$ given by (\ref{CandR}). It should then be the case that
this disk in
contained in $\bbD$. This is in fact so, since
\begin{eqnarray*}
1-|c|-r
&=&\frac{1-|z_1|^4|w_1|^2+|z_1|^4|w_1|-|w_1|-|z_1|^2+|z_1|^2|w_1|^2}{1-|z_1|^4|w_1|^2}\\
&=&\frac{(1-|z_1|^2)(1-|w_1|)(1-|z_1|^2|w_1|)}{1-|z_1|^4|w_1|^2}>0.
\end{eqnarray*}
}\end{remark}

\section{Interpolation bodies associated with a set of interpolants}
\label{S:bodies}

Let us consider the classical
    (unconstrained)
    matrix-valued interpolation problem:
    \medskip

    {\bf MNP:}  {\em Given points $z_{1}, \dots, z_{n} \in {\mathbb
    D}$ and matrices $W_{1}, \dots, W_{n} \in {\mathbb C}^{k \times
    k}$, find $S \in \cS^{k \times k}$
    satisfying interpolation conditions
    \begin{equation}  \label{op-int'}
     S(z_{j}) = W_{j}\quad\text{for}\quad j = 1, \dots, n.
    \end{equation}
    and assume that the associated Pick matrix
    $$
    \Pick = \left[\frac{ I - W_{i} W_{j}^{*}}{1 - z_{i}
    \overline{z_{j}}} \right]_{i,j=1,\dots, n}
    $$
    is invertible.
    Choose a point $z_{0} \ne z_{1}, \dots, z_{n}$
    and consider the problem of characterizing the associated {\em
    interpolation body
    $$
    {\mathfrak B} = {\mathfrak B}({\mathfrak D}, \cS^{k \times k},
    z_{0}):= \{ \tilX = S(z_{0}) \colon S \in \cS^{k \times k}
    \text{ satisfies } \eqref{op-int'}\}.
    $$
    Then application of the classical Pick matrix condition to the
    $(n+1)$-point set $\{z_{1}, z_{2}, \dots, z_{n}, z_{0}\}$ leads
    to the characterization of the Pick body as the set of all
    $\tilX$ for which \eqref{criterion1} is satisfied, where we take
    $$
    \tilE = \begin{bmatrix} I_{k} \\ \vdots \\ I_{k} \end{bmatrix}
    (1-|z_{0}|^{2})^{1/2}, \quad
    \tilW = \begin{bmatrix} W_{1} \\ \vdots \\ W_{n} \end{bmatrix}
    (1-|z_{0}|^{2})^{1/2}.
    $$
    Hence as an application of Theorem \ref{T:DFK} we arrive at the
    well known result (see \cite[Section 5.5]{DFK92} for the case of
    the Schur problem) that interpolation bodies associated with
    Schur-class interpolants for a data set ${\mathfrak D}$
    can be described as matrix  balls.

    We now consider the interpolation body for constrained
    Schur-class interpolants for a data set ${\mathfrak D}$:
    $$
    {\mathfrak B} = {\mathfrak B}(({\mathfrak D}, (\cS_{1})^{k \times
    k}, z_{0}) =
    \{ \tilX = S(z_{0}) \colon S \in (\cS_{1})^{k \times k}
       \text{ satisfies } \eqref{op-int'}\}.
    $$
    For simplicity we assume that $k=1$ and $n=1$.  We are thus led
    to the following problem:  {\em Given nonzero points $z_{0}, z_{1}$
    in the unit disk ${\mathbb D}$, describe the set}
    \begin{equation}  \label{1intbody}
    {\mathcal B}_{z_{0}} = \{ w_{0} \colon s \in \cS_{1}, \,
    s(z_{1}) = w_{1}, s(z_{0}) = w_{0} \}.
    \end{equation}
    Given a data set $(z_{1}, w_{1}; z_{0})$ as above along with a
    complex parameter $x$, we introduce auxiliary matrices as follows:
    \begin{align}
     &  \Pick'_{x} = \begin{bmatrix}  1 - |x|^{2} & 0 & 1 -
    \overline{w_{1}} x \\
    0 & 1 - |x|^{2} & \overline{z_{1}}(1 - \overline{w_{1}} x) \\
    1 - w_{1} \overline{x}& z_{1} (1 - w_{1} \overline{x}) &
    \frac{ 1 - |w_{1}|^{2}}{1 - |z_{1}|^{2}} \end{bmatrix},
    \notag \\
     &  E  = \begin{bmatrix} 1 \\ \overline{z_{0}} \\ \frac{1}{ 1 -
    \overline{z_{0}} z_{1}} \end{bmatrix} \delta_{0}^{1/2}, \quad
    W_{x} = \begin{bmatrix} -x \\ -\overline{z_{0}}x \\ -
    \frac{w_{1}}{1 - \overline{z_{0}}z_{1}} \end{bmatrix}
    \delta_{0}^{1/2} \text{ where } \delta_{0}: = 1 - |z_{0}|^{2}.
    \label{matrices}
  \end{align}
  We then define numbers (i.e., $1 \times 1$ matrices) $c_{x}$ and $R_{x}$ by
  \begin{align}
       c_{x} = & -E^{*}( \Pick'_{x} + W_{x} W_{x}^{*})^{-1} W_{x},
    \quad  R_{x} = \ell_{x}^{1/2} r_{x}^{1/2} \text{ where } \notag \\
     &  \ell_{x} = 1 - W_{x}^{*}\left( \Pick_{x}' + W_{x}
      W_{x}^{*}\right)^{-1} W_{x}, \quad r_{x} = 1 - E^{*} \left(
      \Pick_{x}' + W_{x} W_{x}^{*}\right)^{-1} E.
      \label{numbers}
   \end{align}
   The following result gives an indication of how the geometry of the
   interpolation body ${\mathcal B}_{z_{0}}$ is more complicated for the
   constrained case in comparison with the unconstrained case where it
   is simply a disk; specifically, we identify a union of a 1-parameter
   family of disks, where the parameter itself runs over a certain subset
   of a disk, as a subset of ${\mathcal B}_{z_{0}}$. In general we use the
   notation
   $$
   \overline{\mathbb D}(C,R) = \{ w \in {\mathbb C} \colon |w - C|
   \le R \}
   $$
   for the closed disk in the complex plane with center $C \in
   {\mathbb C}$ and radius $R > 0$.

   \begin{proposition} The union of disks
       \begin{equation}  \label{Bz0}
       \bigcup_{x\in\mathbb D(c,r),r_x>0}\overline{\mathbb D}(c_{x}, R_{x})
       \end{equation}
       is a subset of the interpolation body ${\mathcal B}_{z_{0}}$ in \eqref{1intbody}.
       Here $c,r$ are as in \eqref{CandR} and $r_x,c_{x},R_{x}$
       are as in \eqref{numbers}.
       \end{proposition}

       \begin{proof}
    As an application of Theorem \ref{T:2} we see that $w_{0} \in
    {\mathcal B}_{z_{0}}$ if and only if there exists an $x \in
    {\mathbb D}$ so that the Pick matrix
    $$ {\mathbb P}_{x,w_{0}} =
    \begin{bmatrix} \frac{1 - |w_{1}|^{2}}{1 - |z_{1}|^{2}} &
        \frac{1 - w_{1} \overline{w_{0}}}{1 - z_{1}\overline{z_{0}}}
        & 1 - w_{1} \overline{x} & z_{1} (1 - w_{1} x) \\
     \frac{ 1 - \overline{w_{1}}w_{0}} {1- \overline{z_{1}}
     z_{0}} & \frac{1 - |w_{0}|^{2}}{1 - |z_{0}|^{2}} & 1 - w_{0}
     \overline{x} & z_{0}(1 - w_{0} x) \\
     1 - \overline{w_{1}}x & 1 - \overline{w_{0}}x & 1 - |x|^{2} &
     0 \\ \overline{z_{1}}(1 - \overline{w_{1}}x) &
     \overline{z_{0}} (1 - \overline{w_{0}}x) & 0 & 1 - |x|^{2}
     \end{bmatrix}
   $$
   is positive semidefinite. Interchanging the first two rows with
   the last two rows and similarly for the columns brings us to
   $$
   {\Pick}'_{x,w_{0}} =
   \begin{bmatrix} 1 - |x|^{2} & 0 & 1 - \overline{w_{1}}x &
       \delta_{0}^{1/2} (1 - \overline{w_{0}}x)  \\
       0 & 1 - |x|^{2} & \overline{z_{1}} (1 - \overline{w_{1}}x) &
       \delta_{0}^{1/2} \overline{z_{0}} (1 - \overline{w_{0}}x)  \\
       1 - w_{1} \overline{x} & z_{1} (1 - w_{1}\overline{x}) &
       \frac{1 - |w_{1}|^{2}}{1 - |z_{1}|^{2}} & \delta_{0}^{1/2}
       \frac{ 1 - w_{1} \overline{w_{0}}}{1 - \overline{z_{0}} z_{1}}
       \\  \delta_{0}^{1/2} ( 1 - w_{0} \overline{x}) &
       \delta_{0}^{1/2} z_{0}(1 - w_{0} \overline{x}) &
       \delta_{0}^{1/2} \frac{ 1 -\overline{ w_{1}} w_{0}}{1 - z_{0}
       \overline{z_{1}}}  & 1 - |w_{0}|^{2} \end{bmatrix}
  $$
  where we also multiplied the last row and last column by
  $\delta_{0}^{1/2} = (1 - |z_{0}|^{2})^{1/2}$.
  Next observe that we can write $\Pick_{x,w_{0}}$ as
  $$
  \Pick'_{x,w_{0}} = \begin{bmatrix} \Pick'_{x} & E + W_{x}
  \overline{w_{0}} \\ E^{*} + \overline{w_{0}} W_{x}^{*} & 1 - w_{0}
  \overline{w_{0}} \end{bmatrix}
  $$
  where $\Pick'_{x}$, $E$ and $W_{x}$ are as in \eqref{matrices}.
  We conclude that $w_{0}$ is in the interpolation body ${\mathcal
  B}_{z_{0}}$ if and only if there is an $x \in {\mathbb D}$ for
  which the matrix $\Pick'_{x,w_{0}}$ is positive-semidefinite.
  By Corollary \ref{C:disk} we see in particular that $\Pick'_{x}$ is
  positive definite if and only if $x \in \mathbb D(c,r)$
  where $c$ and $r$ are as in \eqref{CandR}.
  For a fixed such $x$ we apply Theorem \ref{T:DFK} (tailored to the
  case $k=1$) to see that then $\Pick'_{x,w_{0}}$ is positive
  semidefinite in case $r_x>0$ and
  $w_{0} \in \overline{\mathbb D}(c_{x},
  R_{x})$.
 \end{proof}

\end{document}